\documentclass[12pt]{amsart}
\usepackage{latexsym,amssymb,amsmath,tikz, xypic}
\usepackage{comment}

\numberwithin{equation}{section}
\newtheorem{theorem}{Theorem}[section]
\newtheorem{lemma}[theorem]{Lemma}
\newtheorem{proposition}[theorem]{Proposition}
\newtheorem{corollary}[theorem]{Corollary}

\newtheorem{question}[theorem]{Question}

\theoremstyle{definition}
\newtheorem{definition}[theorem]{Definition} 
\newtheorem{construction}[theorem]{Construction}
\newtheorem{remark}[theorem]{Remark}
\newtheorem{example}[theorem]{Example}

\def\frk{\mathfrak}               

\def\Phi{{\frk N}}
\def\opn#1#2{\def#1{\operatorname{#2}}} 
\opn\chara{char} 
\opn\length{\ell} 
\opn\pd{pd} 
\opn\rk{rk}
\opn\projdim{proj\,dim} 
\opn\injdim{inj\,dim} 
\opn\rank{rank}
\opn\depth{depth} 
\opn\grade{grade} 
\opn\height{height}
\opn\embdim{emb\,dim} 
\opn\codim{codim}

\opn\Tr{Tr} 
\opn\bigrank{big\,rank}
\opn\superheight{superheight}
\opn\lcm{lcm}
\opn\trdeg{tr\,deg}
\opn\reg{reg} 
\opn\lreg{lreg} 
\opn\ini{in} 
\opn\lpd{lpd}
\opn\size{size}
\opn\mult{mult}
\opn\dist{dist}
\opn\cone{cone}
\opn\lex{lex}
\opn\rev{rev}
\opn\div{div} \opn\Div{Div} \opn\cl{cl} \opn\Cl{Cl}
\opn\Spec{Spec} \opn\Supp{Supp} \opn\supp{supp} \opn\Sing{Sing}
\opn\Ass{Ass} \opn\Min{Min}
\opn\Ann{Ann} \opn\Rad{Rad} \opn\Soc{Soc}
\opn\Syz{Syz} \opn\Im{Im} \opn\Ker{Ker} \opn\Coker{Coker}
\opn\Am{Am} \opn\Hom{Hom} \opn\Tor{Tor} \opn\Ext{Ext}
\opn\End{End} \opn\Aut{Aut} \opn\id{id} \opn\ini{in}

\opn\nat{nat}
\opn\pff{pf}
\opn\Pf{Pf} \opn\GL{GL} \opn\SL{SL} \opn\mod{mod} \opn\ord{ord}
\opn\Gin{Gin}
\opn\Hilb{Hilb}\opn\adeg{adeg}\opn\std{std}\opn\ip{infpt}
\opn\Pol{Pol}
\opn\sat{sat}
\opn\Var{Var}
\opn\Gen{Gen}

\opn\aff{aff} \opn\con{conv} \opn\relint{relint} \opn\st{st}
\opn\lk{lk} \opn\cn{cn} \opn\core{core} \opn\vol{vol}
\opn\link{link} \opn\star{star}
\opn\gr{gr}

\def\pot#1#2{#1[\kern-0.28ex[#2]\kern-0.28ex]}

\opn\dirlim{\underrightarrow{\lim}}
\opn\inivlim{\underleftarrow{\lim}}
\begin{document}

 
\title{Homological invariants of Cameron--Walker Graphs}
\thanks{\today}

\author[T. Hibi]{Takayuki Hibi}
\address{Department of Pure and Applied Mathematics, Graduate School
of Information Science and Technology, Osaka University, Suita, Osaka
565-0871, Japan}
\email{hibi@math.sci.osaka-u.ac.jp}

\author[H. Kanno]{Hiroju Kanno}
\address{Department of Pure and Applied Mathematics, Graduate School
of Information Science and Technology, Osaka University, Suita, Osaka
565-0871, Japan}
\email{u825139b@ecs.osaka-u.ac.jp}

\author[K. Kimura]{Kyouko Kimura}
\address{Department of Mathematics, 
Faculty of Science, 
Shizuoka University, 836 Ohya, Suruga-ku, Shizuoka 422-8529, Japan}
\email{kimura.kyoko.a@shizuoka.ac.jp}

\author[K. Matsuda]{Kazunori Matsuda}
\address{Kitami Institute of Technology, Kitami, Hokkaido 090-8507, Japan}
\email{kaz-matsuda@mail.kitami-it.ac.jp}
 
\author[A. Van Tuyl]{Adam Van Tuyl}
\address{Department of Mathematics and Statistics\\
McMaster University, Hamilton, ON, L8S 4L8, Canada}
\email{vantuyl@math.mcmaster.ca}

 
\begin{abstract}
Let $G$ be a finite simple connected graph on $[n]$ and $R = K[x_1, \ldots, x_n]$ the polynomial ring in $n$ variables over a field $K$.  The edge ideal of $G$ is the ideal $I(G)$ of $R$ which is generated by those monomials $x_ix_j$ for which $\{i, j\}$ is an edge of $G$.  In the present paper, the possible tuples $(n, \depth (R/I(G)), \reg (R/I(G)), \dim R/I(G), \deg h(R/I(G)))$, where $\deg h(R/I(G))$ is the degree of the $h$-polynomial of $R/I(G)$, arising from Cameron--Walker graphs on $[n]$ will be completely determined.
\end{abstract}

\maketitle


\section*{Introduction}
Among the current trends of commutative algebra,
the role of combinatorics is distinguished.
In particular, the combinatorics of finite simple graphs has created
fascinating research projects in commutative algebra.
Let $G$ be a finite simple graph on the vertex set
$[n] = \{1, \ldots, n\}$
and let $E(G)$ be the set of edges of $G$.
Let $R = K[x_1, \ldots, x_n]$ denote the polynomial ring in $n$ variables
over a field $K$.
The edge ideal of $G$, denoted by $I(G)$, is the ideal of $R$
generated by those monomials $x_i x_j$ with $\{ i, j \} \in E(G)$.

\par
In the present paper, we focus on the invariants
$\depth (R/I(G))$, $\reg (R/I(G))$, $\dim R/I(G)$,
and $\deg h(R/I(G))$, where $\reg (R/I(G))$ denotes
the (Castelnuovo--Mumford) regularity of $R/I(G)$ and
$\deg h(R/I(G))$ denotes the degree of the $h$-polynomial of $R/I(G)$;
see Section \ref{sec:background} for the definitions. 
The relations among these have been studied, for example,
in \cite{HKM, HKMT, HKMVT, HMVT, KumarKumarSarkar, Rinaldo, SeyedFakhariYassemi}.
In particular, Kumar, Kumar, and Sarkar
(\cite[Theorem 4.13]{KumarKumarSarkar}) proved that,
in general, there is no relation between $\depth (R/I(G))$ and $\reg (R/I(G))$,
or $\depth (R/I(G))$ and $\deg h(R/I(G))$.
However the inequality
\begin{displaymath}
  \deg h(R/I(G)) - \reg (R/I(G))
  \leq \dim R/I(G) - \depth (R/I(G))
\end{displaymath}
holds (\cite[Corollary B.4.1]{V2}).
(In fact, this inequality holds for any homogeneous ideal of $R$.) 
Also the first, and last two authors proved,
in \cite[Theorem 13]{HMVT}, the inequality 
\begin{displaymath}
  \reg (R/I(G)) + \deg h(R/I(G)) \leq n. 
\end{displaymath}
In the previous paper \cite{HKMVT},
motivated by the above relation among the regularity,
the degree of the $h$-polynomial,
and the number of vertices of $G$,
the authors investigated the
possible tuples $(\reg (R/I(G)), \deg h(R/I(G)))$ as one varies over
all  connected graphs $G$ on a fixed number of vertices.
In particular, it was shown in \cite{HKMVT} that when we restrict to
the family of 
Cameron--Walker graphs, one can completely determine all such pairs.
(We will recall the definition of a Cameron--Walker graph
in Section \ref{prelim-sec}; see also \cite{HHKO, HKMT,TNT}.) 
Furthermore, in \cite{HKMT} it was proved that
\begin{displaymath}
  \depth (R/I(G)) \leq \reg (R/I(G)) \leq \dim R/I(G)
  = \deg h(R/I(G)) 
\end{displaymath}
for any Cameron--Walker graph $G$. 

In the present paper, we first focus on the relation between
$\depth (R/I(G))$ and $\dim (R/I(G))$ and investigate the
possible ordered pairs $(\depth (R/I(G)), \dim R/I(G))$ arising from
connected graphs $G$ on a  fixed number of vertices
(Section \ref{sec:DepthDimGraph}).
In particular, when we restrict to
Cameron--Walker graphs, we completely determine such pairs
(Theorem \ref{CWdd}).
Moreover, when we restrict to the case that $G$ is a Cameron--Walker graph,
we completely determine all the possible sequences
\begin{displaymath}
  (|V(G)|, \depth (R/I(G)), \reg (R/I(G)), \dim R/I(G), \deg h(R/I(G))), 
\end{displaymath}
which is the main result of the paper (see Theorem \ref{maintheorem3}).

Our paper is organized as follows.
First, the required background is briefly summarized
in Section \ref{sec:background}.
Section \ref{sec:DepthDimGraph} is devoted to finding
the possible pairs $(\depth (R/I(G)), \dim R/I(G))$
arising from finite simple connected graphs $G$ on $[n]$.
Background information on Cameron--Walker graphs
is presented in Section \ref{prelim-sec}.
Also, all possible tuples $(\depth (R/I(G)), \dim R/I(G))$
arising from Cameron--Walker graphs $G$ on $[n]$ are completely determined.
The highlight of this paper is Section \ref{sec:MainResult},
where Theorem \ref{maintheorem3} is finally proved. 

\begin{center}
{\bf Acknowledgements} 
\end{center}

Hibi, Kimura, and Matsuda's research was supported by JSPS KAKENHI 
19H00637, 15K17507 and 20K03550. 
Van Tuyl's research was supported by NSERC Discovery Grant 2019-05412. 
This work was supported by the Research Institute for Mathematical Sciences, 
an International Joint Usage/Research Center located in Kyoto University.


\section{Background}
\label{sec:background}
We recall some of the relevant prerequisites about
homological invariants, graph theory and edge ideals.  

\subsection{Homological Invariants}
Let $R = K[x_1, \ldots, x_n]$ denote the polynomial ring in $n$ variables
over a field $K$ with $\deg x_i = 1$ for each $i$.  For any ideal
$I$ of $R$, the {\it dimension} of $R/I$, denoted $\dim R/I$, is
the length of the longest chain of prime ideals in $R/I$.  The
{\it depth} of $R/I$, denoted ${\rm depth}(R/I)$, is the length
of the longest regular sequence in $R/I$.

If $I \subseteq R$
is a homogeneous ideal, then the {\it Hilbert series} of $R/I$ is
$$H_{R/I}(t) = \sum_{i \geq 0} \dim_K [R/I]_it^{i}$$
where $[R/I]_i$ denotes the $i$-th graded piece of $R/I$.
If $\dim R/I = d$, then 
the Hilbert series of $R/I$ is the form 
$$
H_{R/I}(t) = \frac{h_{0} + h_{1}t + h_{2}t^{2} + \cdots + h_{s}t^{s}}{(1 - t)^d}
= \frac{h_{R/I}(t)}{(1-t)^d}, 
$$
where each $h_{i} \in \mathbb{Z}$ (\cite[Proposition 4.4.1]{BH})
and $h_{R/I}(1) \neq 0$.
We say that 
$$
h_{R/I}(t) = h_{0} + h_{1}t + h_{2}t^{2} + \cdots + h_{s}t^{s}
$$ 
with $h_{s} \neq 0$ is the {\em $h$-polynomial} of $R/I$.  
We put $\deg h(R/I) := \deg h_{R/I}(t)$. 

If $I$ is a homogeneous ideal of $R$, then the
({\em Castelnuovo--Mumford}) {\em regularity} of $R/I$ is given by
$$
{\rm reg}(R/I)=  \max\{j-i ~|~ \beta_{i,j}(R/I) \neq 0\}
$$
where $\beta_{i,j}(R/I)$ denotes the $(i,j)$-th graded Betti number
in the minimal graded free resolution of $R/I$.  (For more details
see, for example,  \cite[Section 18]{Peeva}.)


\subsection{Graph theory}
Let $G = (V(G), E(G))$ be a finite simple graph
(i.e., a graph with no loops and no multiple edges) on the vertex set 
$V(G) = \{x_{1}, \ldots, x_{n}\}$ and edge set $E(G)$.
For a subset $W \subset V(G)$, the {\em induced subgraph} of $G$ on $W$,
denoted by $G_W$, is the graph whose vertex set is $W$ and
edge set is the set of all edges of $G$ contained in $W$.
For a vertex $x_i \in V(G)$, we denote by $N_G (x_i)$,
the set of all neighbours of $x_i$ in $G$:
\begin{displaymath}
  N_G (x_i) = \{ x_j \in V(G) \; | \; \{ x_i, x_j \} \in E(G) \}.
\end{displaymath}
A subset $S \subset V(G)$ is an {\em independent set} of $G$ if
$\{x_{i}, x_{j}\} \not\in E(G)$ for all $x_{i}, x_{j} \in S$. 
In particular, the empty set $\emptyset$ is an independent set of $G$. 

The  $S$-suspension (\cite[p.313]{HKM}) of a graph $G$ will play
an important role in Section \ref{sec:DepthDimGraph};
we recall this construction.
Let $G = (V(G),E(G))$ be a finite simple graph. 
For any independent set $S \subset V(G) = \{x_1,\ldots,x_n\}$,
we construct a graph $G^{S}$ with vertex and edge set given by:
\begin{enumerate}
\item[$\bullet$] $V(G^{S}) = V(G) \cup \{x_{n + 1}\}$, where $x_{n + 1}$
  is a new vertex, and
\item[$\bullet$] $E(G^{S}) = E(G) \cup \left\{ \{x_{i}, x_{n + 1}\} ~|~ x_{i} \not\in S \right\}.$
  \end{enumerate}
That is, we add a new vertex $x_{n+1}$ and join it to every vertex {\it not}
in $S$.  The graph $G^{S}$ is called the {\em $S$-suspension} of $G$.  
We note that this construction still holds in the case of $S = \emptyset$.

\subsection{Edge ideals}  
In this subsection, we define the edge ideal of a finite simple graph. 
Let $G=(V(G),E(G))$ be a finite simple graph on $V(G) = \{x_1,\ldots,x_n\}$. 
We associate with $G$ the quadratic square-free monomial ideal
$$
I(G) = ( x_{i}x_{j} ~|~ \{x_{i}, x_{j}\} \in E(G) ) \subseteq
R = K[x_1,\ldots,x_n]. 
$$
The ideal $I(G)$ is called the {\it edge ideal} of the graph $G$.

It is known that the dimension of $R/I(G)$ is given by  
\begin{lemma}
\label{dim}
$$\dim R/I(G) =  \max\left.\left\{ |S| ~\right|~ S \ \text{is an independent set of}\  G \right\}.$$   
\end{lemma}


Let $G$ be a finite simple graph, and let $S \subset V(G)$ be an independent set of $G$. 
The following lemma describes the relationship between homological invariants of 
$I(G)$ and $I(G^{S})$.

\begin{lemma}\label{S-suspension}
  Let $G$ be a finite simple graph on $V(G) = \{x_1,\ldots,x_n\}$ and
  $G^S$ the $S$-suspension of $G$ for some independent set $S$ of $G$.  
  If $I(G) \subseteq R =K[x_1,\ldots,x_n]$ and
  $I(G^S) \subseteq  R' = K[x_1,\ldots,x_n,x_{n+1}]$
  are the respective edge ideals,
  then
  \begin{enumerate}
	\item[$(i)$] $\dim R'/I(G^{S}) = \dim R/I(G)$ if $|S| \leq \dim R/I(G) - 1$. 
	\item[$(ii)$] ${\rm depth}(R'/I(G^{S})) = {\rm depth}(R/I(G))$ if $|S| = {\rm depth}(R/I(G)) - 1$. 
	\item[$(iii)$] ${\rm depth}(R'/I(G^{S})) = 1$ if $S = \emptyset$. 
\end{enumerate} 
\end{lemma}

\begin{proof}
By virtue of \cite[Lemma 1.5]{HKM}, we have (i). 

We show (ii) and (iii). Let us consider the following short exact sequence:
\[
0 \to R'/(I(G^{S}):(x_{n + 1})) \ (-1) \xrightarrow{\times x_{n + 1}} R'/I(G^{S}) \to R'/(I(G^{S}) + (x_{n + 1})) \to 0. 
\]
First note that $(I(G^S):(x_{n+1}))=(x_i~|~ x_i \not\in S\cup \{x_{n+1}\})$.  
Then it follows that ${\rm depth}(R'/(I(G^{S}):(x_{n + 1}))) = |S| + 1$
and ${\rm depth}(R'/(I(G^{S}) + (x_{n + 1}))) = {\rm depth}(R/I(G))$ 
since $R'/(I(G^{S}):(x_{n + 1})) \cong K\left[x_{i} ~|~ x_{i} \in S \cup \{x_{n + 1}\}\right]$ and $R'/(I(G^{S}) + (x_{n + 1})) \cong R/I(G)$.  
By virtue of the Depth Lemma as given in \cite[Lemma 3.1.4]{V},
we have the desired conclusion.
\end{proof}

\par
As an easy application of Lemmas \ref{dim} and \ref{S-suspension},
we compute the dimension and the depth of the edge ideal of a star graph.
Recall that a star graph is  a graph joining some paths of length $1$
at one common vertex (see Figure \ref{fig:star}).

\begin{figure}[htbp]
  \centering
\bigskip

\begin{xy}
	\ar@{} (0,0);(80, -16)   *\cir<4pt>{} = "C"
	\ar@{-} "C";(60, 0)  *\cir<4pt>{} = "D";
	\ar@{-} "C";(70, 0)  *\cir<4pt>{} = "E";
	\ar@{} "C"; (80, 0) *++!U{\cdots}
	\ar@{-} "C";(100, 0)  *\cir<4pt>{} = "F";
\end{xy}

\bigskip

  \caption{The star graph}
  \label{fig:star}
\end{figure}
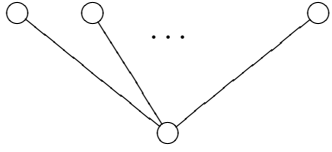

\begin{proposition}
  \label{star}
  Let $G$ be a star graph with $n \geq 2$ vertices. Then
  \begin{displaymath}
    \dim R/I(G) = n-1 ~\mbox{and}~ \depth (R/I(G)) = 1.
  \end{displaymath}
\end{proposition}
\begin{proof}
  We set $V(G) = \{ x_1, \ldots, x_n \}$ and
  $E(G) = \{ \{ x_i, x_n \} ~|~ i = 1, \ldots, n-1 \}$.
  Then the set $\{ x_1, \ldots, x_{n-1} \} \subset V(G)$ is independent with
  maximal cardinality. Hence Lemma \ref{dim} implies
  $\dim R/(G) = n-1$.

  \par
  Let $G_0$ be the graph consisting of $n-1$ isolated vertices
  $x_1, \ldots, x_{n-1}$.
  Then $G$ coincides with the $S$-suspension of $G_0$ with $S = \emptyset$.
  Hence Lemma \ref{S-suspension} (iii) implies
  $\depth (R/I(G)) = 1$.
\end{proof}


\section{The set ${\rm Graph}_{{\rm depth}, \dim}(n)$}
\label{sec:DepthDimGraph}
Let ${\rm Graph}(n)$ denote the set of all finite simple connected graphs
with $n$ vertices and 
let ${\rm Graph}_{{\rm depth}, \dim}(n)$ denote the set of all possible pairs
$(\depth(R/I(G)), \dim R/I(G))$ arising from $G \in {\rm Graph}(n)$, i.e., 
\[
{\rm Graph}_{{\rm depth}, \dim}(n) = \{(\depth(R/I(G)), \dim R/I(G))
\, | \, G \in {\rm Graph}(n) \}. 
\]
In this section we discuss which pairs $(a,b) \in \mathbb{N}^2$ are elements of
${\rm Graph}_{{\rm depth}, \dim}(n)$. 
Note that if $(a,b) \in {\rm Graph}_{{\rm depth}, \dim}(n)$, then
$1 \leq a \leq b \leq n$ because of the well-known inequality
$1 \leq \depth (R/I(G)) \leq \dim R/I(G) \leq n$
for any finite simple graph $G$ on $[n]$.

\par
We can easily compute ${\rm Graph}_{{\rm depth}, \dim}(n)$ when $n$ is small. 
\begin{example}
  \label{DD(123)}
  We compute ${\rm Graph}_{{\rm depth}, \dim}(n)$ when $n=1,2,3$:
  \begin{enumerate}
  \item ${\rm Graph}_{{\rm depth}, \dim}(1) = \{ (1, 1) \}$.
    Actually, a graph $G$ with one vertex consists of one isolated vertex.
    Hence $I(G) = (0) \subset R = K[x_1]$
    and $\depth (R/I(G)) =\dim R/I(G) = 1$. 
  \item ${\rm Graph}_{{\rm depth}, \dim} (2) = \{ (1, 1) \}$.
    Indeed, a connected graph $G$ with two vertices consists
    of two vertices connected by an edge. 
    Hence $I(G) = ( x_1 x_2 ) \subset R = K[x_1, x_2]$.
    Then $\depth (R/I(G)) = \dim R/I(G) = 1$. 
  \item ${\rm Graph}_{{\rm depth}, \dim}(3) = \{(1, 1), (1, 2)\}$.
    A connected graph $G$ with three vertices is either
    a line graph with three vertices or a triangle. 
    In the former (resp.\ latter) case,
    we can write $I(G) = (x_{1}x_{2}, x_{2}x_{3}) \subset R$
    (resp.\  $I(G) = (x_{1}x_{2}, x_{1}x_{3}, x_{2}x_{3}) \subset R$)
    where
    $R = K[x_1, x_2, x_3]$. 
    Then $\depth (R/I(G)) = 1$, $\dim R/I(G) = 2$
    (resp.\  $\depth (R/I(G)) = \dim R/I(G) = 1$). 
  \end{enumerate}
\end{example}


\par
We give some observations on ${\rm Graph}_{{\rm depth}, \dim}(n)$.
The first one is an easy consequence of Proposition \ref{star}. 

\begin{lemma}
  \label{(1,n-1)}
  For all $n \geq 2$, we have
  $(1, n-1) \in {\rm Graph}_{{\rm depth}, \dim}(n)$.
\end{lemma}

The second one relies heavily on the $S$-suspension construction.
\begin{lemma}\label{BasicLemma}
  For all $n \geq 1$, we have
  ${\rm Graph}_{{\rm depth}, \dim}(n) \subseteq {\rm Graph}_{{\rm depth}, \dim}(n+1)$. 
\end{lemma}

\begin{proof}
  Assume that $(a, b) \in {\rm Graph}_{{\rm depth}, \dim}(n)$. 
  Then there exists $G \in {\rm Graph}(n)$ with ${\rm depth}(R/I(G)) = a$ and
  $\dim R/I(G) = b$. Let $S$ be an independent set
  of $G$ with $|S| = {\rm depth}(R/I(G)) - 1 = a-1$.  This
  is possible, since ${\rm depth}(R/I(G)) \geq 1$ for all graphs $G$,
  and since there is an independent set
  $W$ with $|W| =\dim R/I(G) > {\rm depth}(R/I(G)) -1$.
  By virtue of Lemma \ref{S-suspension}, we have
  ${\rm depth}(R'/I(G^{S})) = a$ and $\dim R'/I(G^{S}) = b$.
  Hence $(a, b) \in {\rm Graph}_{{\rm depth}, \dim}(n + 1)$ since $|V(G^S)| = n + 1$. 
 \end{proof}

Third, we prove the following fact about ${\rm Graph}_{{\rm depth}, \dim}(n)$.
\begin{lemma}\label{DD}
  Let $a, b$ be integers with $1 \leq a \leq b$.
  Then $(a, b) \in {\rm Graph}_{{\rm depth}, \dim}(a+b)$. 
\end{lemma}

In order to prove Lemma \ref{DD}, 
we introduce a new class of graphs. 
\begin{construction}
\label{const:G(m;s)}
  Let $m \geq 1$ and $1 \leq s_{1} \leq s_{2} \leq \cdots \leq s_{m}$ 
  be positive integers. 
  Then we define the graph $G(m; s_{1}, s_{2}, \ldots, s_{m})$ as follows: 
\begin{itemize}
	\item $\displaystyle V \left(G(m; s_{1}, s_{2}, \ldots, s_{m}) \right) = \{ v_{1}, \ldots, v_{m} \} \cup \bigcup_{i = 1}^{m} \left\{ x_{1}^{(i)}, \ldots, x_{s_{i}}^{(i)} \right\}$; 
	\item $E \left(G(m; s_{1}, s_{2}, \ldots, s_{m}) \right) = \left\{ \{v_{i}, v_{j}\} : 1 \leq i < j \leq m \right\}$ \\ 
	$$\ \ \ \ \ \ \ \ \ \ \ \ \ \cup \bigcup_{1 \leq k \leq m} \left\{ \{v_{k}, x_{1}^{(k)}\}, \ldots, \{v_{k}, x_{s_{k}}^{(k)}\} \right\}. $$
\end{itemize}
\end{construction}

Namely, the graph $G(m; s_{1}, s_{2}, \ldots, s_{m})$
is the finite connected graph consisting of 
the complete graph on $\{ v_{1}, \ldots, v_{m} \}$ 
such that each $v_i$ has $s_{i}$ leaf edges,
where a leaf edge is an edge having a vertex of degree one. 
Note that $G(1; s_{1})$ is a star graph. 

\begin{example}
Let $m = 3, s_{1} = 1, s_{2} = 2$ and $s_{3} = 3$. Then the graph $G(3; 1,2,3)$ is as in Figure \ref{fig:G(3; 1,2,3)}: 

\bigskip

  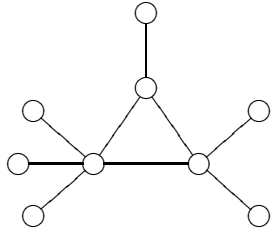
\begin{figure}[htbp]
  \centering
\begin{xy}
	\ar@{} (0, 0);(75, 0) *\cir<4pt>{} = "A";
	\ar@{-} "A";(68, -10) *\cir<4pt>{} = "B";
	\ar@{-} "A";(82, -10) *\cir<4pt>{} = "C";
	\ar@{-} "B";"C";
	\ar@{-} "A";(75, 10) *\cir<4pt>{};
	\ar@{-} "B";(58, -10) *\cir<4pt>{};
	\ar@{-} "B";(60, -3) *\cir<4pt>{};
	\ar@{-} "B";(60, -17) *\cir<4pt>{};
	\ar@{-} "C";(90, -3) *\cir<4pt>{};
	\ar@{-} "C";(90, -17) *\cir<4pt>{};
\end{xy}

\bigskip

  \caption{The graph $G(3; 1,2,3)$}
  \label{fig:G(3; 1,2,3)}
\end{figure}
\end{example}

We investigate the invariants for $G(m; s_1, s_2, \ldots, s_m)$. 

\begin{proposition}\label{G(m)}
Let $m \geq 1$ and $1 \leq s_{1} \leq s_{2} \leq \cdots \leq s_{m}$ be positive integers and 
$G = G(m; s_{1}, s_{2}, \ldots, s_{m})$ the graph as
in Construction \ref{const:G(m;s)}. Then $|V(G)| = m + \sum_{i = 1}^{m}s_{i}$ and 
\begin{enumerate}
  \item[$(i)$] $\displaystyle \dim R/I(G) = \sum_{i = 1}^{m}s_{i}$. 
  \item[$(ii)$] $\displaystyle {\rm depth} (R/I(G)) = 1 + \sum_{i = 1}^{m - 1}s_{i}$. 
\end{enumerate}
\end{proposition}

\begin{proof}
  In order to prove (i), we use Lemma \ref{dim}.
  Suppose that there exists an independent set $S \subseteq V(G)$ of $G$
  with $|S| > \sum_{i=1}^{m}s_{i}$. 
  Since $|S| > \sum_{i=1}^{m}s_{i}$,
  one of $v_1, \ldots, v_m$ is contained in $S$.
  We first assume that $v_{1} \in S$. 
  As $S$ is an independent set of $G$, one has
  $x^{(1)}_{1}, \ldots, x^{(1)}_{s_{1}}, v_{2}, \ldots, v_{m} \not\in S$.
  Then $|S| \leq |V(G)| - s_{1} - (m - 1) = \sum_{i=1}^{m}s_{i} - s_{1} + 1
  \leq \sum_{i=1}^{m}s_{i}$, a contradiction.
  Similarly we have a contradiction if $v_i \in S$ for $i = 2, \ldots, m$. 
  Hence there is no independent set with cardinality more than
  $\sum_{i=1}^{m}s_{i}$. 
  Because $V(G) \setminus \{v_{1}, \ldots, v_{m}\}$ is an independent set
  of $G$ of size $\sum_{i=1}^m s_i$,  we have
  $\dim R/I(G) = \sum_{i=1}^{m}s_{i}$ by virtue of Lemma \ref{dim}. 

  \par
  We prove (ii) by induction on $m$. 
  If $m = 1$, then $G$ is a star graph
  joining $s_{1}$ paths of length one at one common vertex $v_{1}$.
  Hence ${\rm depth}(R/I(G)) = 1$ by Proposition \ref{star}. 

Next, we assume that $m > 1$. Then one has 
\begin{eqnarray*}
I(G) + (v_{m}) &=& (v_{m}) + I(G(m-1; s_{1}, s_{2}, \ldots, s_{m - 1})), \\
I(G) : (v_{m}) &=& (v_{1}, v_{2}, \ldots, v_{m - 1}) + \left(x_{1}^{(m)}, x_{2}^{(m)}, \ldots, x_{s_{m}}^{(m)}\right). 
\end{eqnarray*}
Hence it follows that 
\begin{eqnarray*}
& & R/(I(G) + (v_{m})) \\ 
&\cong& K\left[x_{j}^{(m)} : 1 \leq j \leq s_{m}\right] \otimes_{K} K\left[V(G(m-1; s_{1}, \ldots, s_{m - 1}))\right] / I(G(m-1; s_{1}, \ldots, s_{m - 1})),
\end{eqnarray*}
and
\[
R/(I(G) : (v_{m})) 
\cong \bigotimes_{i = 1}^{m - 1} K\left[x_{j}^{(i)} : 1 \leq j \leq s_{i}\right] \otimes_{K} K[v_{m}]. 
\]
Thus by induction we have
\begin{eqnarray*}
{\rm depth}(R/(I(G) + (v_{m}))) &=& 1 + \sum_{i = 1}^{m - 2}s_{i} + s_{m}, ~~\mbox{and}\\
{\rm depth}(R/(I(G) : (v_{m}))) &=& 1 + \sum_{i = 1}^{m - 1}s_{i}. 
\end{eqnarray*}
Hence ${\rm depth}(R/I(G) + (v_{m})) \geq {\rm depth}(R/I(G) : (v_{m}))$ since $s_{m} \geq s_{m - 1}$. 
Thus, by applying the Depth Lemma (\cite[Lemma 3.1.4]{V}) to the short exact sequence
\[
0 \to R/(I(G) : (v_{m})) \ (-1) \xrightarrow{\times v_{m}} R/I(G) \to R/(I(G) + (v_{m})) \to 0, 
\]
one has $\displaystyle {\rm depth}(R/I(G)) = {\rm depth}(R/I(G) : (v_{m})) = 1 + \sum_{i = 1}^{m - 1}s_{i}$. 
\end{proof}

Lemma \ref{DD} is now easily derived by using the graph
in Construction \ref{const:G(m;s)}. 

\begin{proof}[{Proof of Lemma \ref{DD}}]
  The graph $G = G(a; 1,1, \ldots, 1, b - a + 1)$ has $|V(G)| = a + b$.
  Then Proposition \ref{G(m)} says that 
${\rm depth}(R/I(G)) = a$ and $\dim R/I(G) = b$. 
\end{proof}

Lemmas \ref{BasicLemma} and \ref{DD} imply
that for any pair $(a,b)$ of integers with
$1 \leq a \leq b$, one has $(a,b) \in {\rm Graph}_{{\rm depth}, \dim}(n)$
for all $n \gg 0$. However for fixed $n$, we have the following theorem.
Note that when $G \in {\rm Graph}_{{\rm depth},\dim} (n)$ has an edge,
then $n \geq 2$ and $\dim R/I(G) \leq n-1$.

\begin{theorem}\label{LatticePointsDD}
  Let $1 \leq a \leq b$ and $n \geq 2$ be integers. 
  Assume that $b \leq n - 1$.
  If $a \leq b + 1 - \left\lceil\frac{b}{n - b}\right\rceil$,
  then $(a, b) \in {\rm Graph}_{{\rm depth}, \dim}(n)$.
\end{theorem}

\begin{proof}
  Assume that $n - b \geq a$. Then $a + b \leq n$. 
  By virtue of Lemma \ref{BasicLemma} together with Lemma \ref{DD},
  we have $(a, b) \in {\rm Graph}_{{\rm depth}, \dim}(a + b)
  \subseteq {\rm Graph}_{{\rm depth}, \dim}(n)$.
    
  \par
  Next, we assume that $n - b < a$.
  Then $b<n-1$ because if $b=n-1$, then $n-b < a$ implies $1<a$, and
  $a \leq b + 1 - \left\lceil\frac{b}{n - b}\right\rceil$ implies $a \leq 1$,
  a contradiction. 
  

  \par
  We write $a-1 = q (n-b-1) + t$ where $q$ and $t$
  are integers and $0 \leq t < n-b-1$.
  Note that $q \geq 1$.
  If $t = 0$, then we set $s_1 = \cdots = s_{n-b-1} = q$. 
  If $t \neq 0$, then we set $s_1 = \cdots = s_{n-b-1-t} = q$
  and $s_{n-b-t} = \cdots = s_{n-b-1} = q+1$.
  We note that in each case one has $\sum_{i = 1}^{n - b - 1} s_{i} = a - 1$.
  Also set $s_{n-b} = b-a+1$. 
  We claim that $s_{n-b} \geq s_{n-b-1}$. Indeed
  since
  \begin{displaymath}
    a \leq b+1 - \left\lceil\frac{b}{n - b}\right\rceil
    \leq b+1 - \frac{b}{n-b}, 
  \end{displaymath}
  one has $a-1 \leq b(n-b-1)/(n-b)$. Hence $b/(n-b) \geq (a-1)/(n-b-1)$.
  Therefore 
  \begin{displaymath}
      s_{n - b}
      = b - a + 1 \geq \left\lceil\frac{b}{n - b}\right\rceil
    \geq \left\lceil\frac{a - 1}{n - b - 1}\right\rceil = s_{n-b-1}. 
  \end{displaymath}
  Let us consider the 
  graph $G = G(n - b; s_{1}, \ldots, s_{n - b - 1}, s_{n - b})$. 
  Then $|V(G)| = n$, and Proposition \ref{G(m)} says that
  ${\rm depth}(R/I(G)) = a$ and $\dim R/I(G) = b$.
  Hence $(a, b) \in {\rm Graph}_{{\rm depth}, \dim}(n)$.

\end{proof}

\par
For $n \geq 3$, we set
\[
C^{-}(n) := \{ (1,n-1) \} \cup
\left\{ (a, b) \in \mathbb{N}^2 \ \middle| \
a \leq b, \, 1 \leq a \leq \left\lfloor \frac{n}{2} \right\rfloor,
\, 1 \leq b \leq n - 2 \right\} \subseteq \mathbb{N}^{2}
\]
and
\[
C^{+}(n) := \left\{ (a, b) \in \mathbb{N}^2 \ \middle| \
1 \leq a \leq b \leq n - 1 \right\} \subseteq \mathbb{N}^{2}. 
\]
The following theorem is the main result of this section, which says 
that the set ${\rm Graph}_{{\rm depth}, \dim}(n)$ is sandwiched
by these convex lattice polytopes. 
\begin{theorem}
\label{maintheorem1}
For all $n \geq 3$, $C^{-}(n) \subseteq {\rm Graph}_{{\rm depth}, \dim}(n) \subseteq C^{+}(n)$. 
\end{theorem}

We use the following lemma to prove Theorem \ref{maintheorem1}. 
\begin{lemma}
\label{(*)}
Let $n \geq 6$ be an integer. 
If $\left\lceil\frac{n}{2}\right\rceil + 1 \leq b \leq n - 2$, then
\begin{equation}\label{inequality}
    b + 1 - \left\lceil\frac{b}{n - b}\right\rceil \geq \left\lfloor \frac{n}{2} \right\rfloor.
\end{equation}
\end{lemma}
\begin{proof}
  Set $b = \left\lceil\frac{n}{2}\right\rceil + c$. 
  Then $1 \leq c \leq \left\lfloor \frac{n}{2} \right\rfloor - 2$, and 
  \begin{displaymath}
    \begin{aligned}
      b + 1 - \left\lceil\frac{b}{n - b}\right\rceil
      &= \left\lceil \frac{n}{2} \right\rceil + c + 1
      - \left\lceil \frac{ \left\lceil \frac{n}{2} \right\rceil + c}{\left\lfloor \frac{n}{2} \right\rfloor - c} \right\rceil \\
      &= \left\{
      \begin{alignedat}{3}
        &\frac{n}{2} + c + 1
         - \left\lceil \frac{n+2c}{n-2c} \right\rceil,
         &\quad &\text{if $n$ is even}, \\
        &\frac{n+1}{2} + c + 1
         - \left\lceil \frac{n+1+2c}{n-1-2c} \right\rceil,
        &\quad &\text{if $n$ is odd}. 
      \end{alignedat}
      \right.
    \end{aligned}
  \end{displaymath}
  First, we assume that $n$ is even. Then
  \begin{displaymath}
    \begin{aligned}
      b + 1 - \left\lceil\frac{b}{n - b}\right\rceil
      - \left\lfloor\frac{n}{2}\right\rfloor 
      &= \frac{n}{2} + c+1 - \left\lceil\frac{n+2c}{n-2c}\right\rceil
      - \frac{n}{2} \\
      &= c+1 + \left\lfloor -\frac{n+2c}{n-2c} \right\rfloor \\
      &= \left\lfloor \frac{(c+1)(n-2c)-(n+2c)}{n-2c} \right\rfloor \\
      &= \left\lfloor \frac{-2c^2 + (n-4)c}{n-2c} \right\rfloor. 
    \end{aligned}
  \end{displaymath}
  Since $n-2c>0$, in order to prove (\ref{inequality}),
  it is sufficient to show $-2c^2 + (n-4)c \geq 0$. 

  \par
  Consider the function $f(x) = -2x^{2} + (n - 4)x$.
  Then $f(x)$ is convex-upward and
  $f(x) \geq 0$ for all $0 \leq x \leq (n-4)/2$.
  Since $0 < 1 \leq c \leq \left\lfloor \frac{n}{2} \right\rfloor - 2
  = (n-4)/2$, we have $f (c) \geq 0$, as desired.  

  \par
  Next, we assume that $n$ is odd. Then
  \begin{displaymath}
    \begin{aligned}
      b + 1 - \left\lceil\frac{b}{n - b}\right\rceil
      - \left\lfloor\frac{n}{2}\right\rfloor 
      &= \frac{n+1}{2} + c+1 - \left\lceil\frac{n+1+2c}{n-1-2c}\right\rceil
      - \frac{n-1}{2} \\
      &= c+2 + \left\lfloor -\frac{n+1+2c}{n-1-2c} \right\rfloor \\
      &= \left\lfloor \frac{(c+2)(n-1-2c)-(n+1+2c)}{n-1-2c} \right\rfloor \\
      &= \left\lfloor \frac{-2c^2 + (n-7)c + n-3}{n-1-2c} \right\rfloor. 
    \end{aligned}
  \end{displaymath}
  Since $n-1-2c>0$, in order to prove (\ref{inequality}), 
  it is sufficient to show $-2c^2 + (n-7)c + n-3 \geq 0$. 

  \par
  Consider the function $f(x) = -2x^{2} + (n - 7)x + n-3$.
  Then $f(x)$ is convex-upward. Also $f(1) = 2n - 12 > 0$ and
  $f(\left\lfloor \frac{n}{2} \right\rfloor - 2) = f(\frac{n - 5}{2}) = 2$
  imply $f(x) \geq 0$ for all
  $1 \leq x \leq \left\lfloor \frac{n}{2} \right\rfloor - 2$.
  Since $0 < 1 \leq c \leq \left\lfloor \frac{n}{2} \right\rfloor - 2$,
  we have $f (c) \geq 0$, as desired. 

\end{proof}

Now we are in position to prove Theorem \ref{maintheorem1}. 
\begin{proof}[{Proof of Theorem \ref{maintheorem1}}]
  Take $(a, b) \in C^{-}(n)$.
  Since we know $(1, n-1) \in {\rm Graph}_{{\rm depth}, \dim} (n)$
  by Lemma \ref{(1,n-1)}, we may assume $(a,b) \neq (1, n-1)$.   
  Then $a \leq b$, $1 \leq a \leq \left\lfloor \frac{n}{2} \right\rfloor$
  and $1 \leq b \leq n - 2$. 
  If $a + b \leq n$, then $(a, b) \in {\rm Graph}_{{\rm depth}, \dim}(a + b) \subseteq {\rm Graph}_{{\rm depth}, \dim}(n)$ by virtue of Lemmas \ref{BasicLemma} and
  \ref{DD}. 
  
  \par
  Assume that $a + b \geq n + 1$.
  Since $a \leq \left\lfloor \frac{n}{2} \right\rfloor$ and $b \leq n - 2$,
  one has $\left\lceil \frac{n}{2} \right\rceil + 1 \leq b \leq n - 2$. 
  It then also follows that $n \geq 6$. 
  By virtue of Lemma \ref{(*)},
  we have $b + 1 - \left\lceil\frac{b}{n - b}\right\rceil
  \geq \left\lfloor \frac{n}{2} \right\rfloor \geq a$. 
  Thus one has $(a, b) \in {\rm Graph}_{{\rm depth}, \dim}(n)$
  by Theorem \ref{LatticePointsDD}. 
  Therefore we have $C^{-}(n) \subseteq {\rm Graph}_{{\rm depth}, \dim}(n)$.

  \par
  It is easy to see that ${\rm Graph}_{{\rm depth}, \dim}(n) \subseteq C^{+}(n)$
  since $1 \leq {\rm depth}(R/I(G)) \leq \dim R/I(G) \leq n - 1$
  for all $G \in {\rm Graph}(n)$. 
\end{proof}

We compare $C^{-} (n)$ or $C^{+} (n)$ with ${\rm Graph}_{{\rm depth}, \dim}(n)$ 
for small $n$. 
\begin{example}\label{DD(456789)}
  Using Macaulay2 \cite{GS}, we computed ${\rm Graph}_{{\rm depth},\dim}(n)$
  for $n=4,\ldots,9$.  The results of these computations are summarized
  in Figures \ref{figure456} and \ref{figure789}.
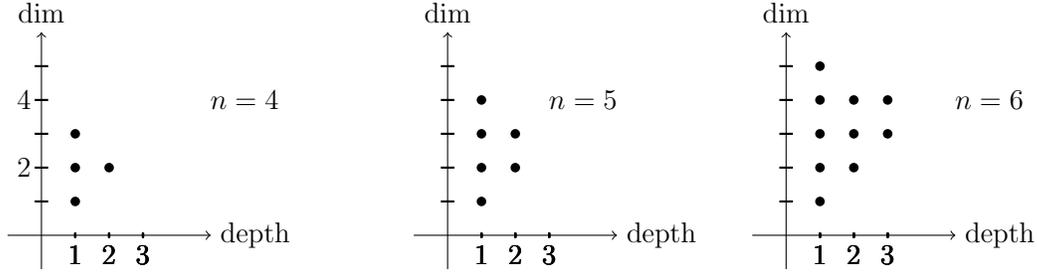
\begin{figure}
\begin{tikzpicture}[thick,scale=0.9, every node/.style={scale=0.9}]
\draw[thin,->] (-4.5,0) -- (-1.5,0) node[right] {${\rm depth}$};
\draw[thin,->] (-4,-0.5) -- (-4,3.0) node[above] {$\dim$};

\foreach \x [count=\xi starting from 0] in {1,2,3,4,5}{
        \draw (-4.1,\x/2) -- (-3.9,\x/2);
    \ifodd\xi
        \node[anchor=east] at (-4,\x/2) {$\x$};
    \fi
 \draw (1/2-4,1pt) -- (1/2-4,-1pt);
 \draw (2/2-4,1pt) -- (2/2-4,-1pt);
 \draw (3/2-4,1pt) -- (3/2-4,-1pt);
     \node[anchor=north] at (1/2-4,0) {$1$};
     \node[anchor=north] at (2/2-4,0) {$2$};
     \node[anchor=north] at (3/2-4,0) {$3$};
}

\foreach \point in {(1/2-4,1/2),(1/2-4,2/2),(1/2-4,3/2),
  (2/2-4,2/2)
}{
    \fill \point circle (2pt);
}

\draw[thin,->] (5.5-4,0) -- (8.5-4,0) node[right] {${\rm depth}$};
\draw[thin,->] (6-4,-0.5) -- (6-4,3.0) node[above] {$\dim$};

\foreach \x [count=\xi starting from 0] in {1,2,3,4,5}{
    \draw (5.9-4,\x/2) -- (6.1-4,\x/2);
    \ifodd\xi
    \fi
\draw (6+1/2-4,1pt) -- (6+1/2-4,-1pt);
 \draw (6+2/2-4,1pt) -- (6+2/2-4,-1pt);
 \draw (6+3/2-4,1pt) -- (6+3/2-4,-1pt);
     \node[anchor=north] at (6+1/2-4,0) {$1$};
     \node[anchor=north] at (6+2/2-4,0) {$2$};
     \node[anchor=north] at (6+3/2-4,0) {$3$};
}

\
\foreach \point in {(6+1/2-4,1/2),(6+1/2-4,2/2),(6+1/2-4,3/2),(6+1/2-4,4/2),
  (6+2/2-4,2/2),(6+2/2-4,3/2)
}
  {
    \fill \point circle (2pt);
}

\draw[thin,->] (5.5+1,0) -- (8.5+1,0) node[right] {${\rm depth}$};
\draw[thin,->] (6+1,-0.5) -- (6+1,3.0) node[above] {$\dim$};

\foreach \x [count=\xi starting from 0] in {1,2,3,4,5}{
    \draw (5.9+1,\x/2) -- (6.1+1,\x/2);
    \ifodd\xi
    \fi
\draw (6+1/2+1,1pt) -- (6+1/2+1,-1pt);
 \draw (6+2/2+1,1pt) -- (6+2/2+1,-1pt);
 \draw (6+3/2+1,1pt) -- (6+3/2+1,-1pt);
     \node[anchor=north] at (6+1/2+1,0) {$1$};
     \node[anchor=north] at (6+2/2+1,0) {$2$};
     \node[anchor=north] at (6+3/2+1,0) {$3$};
}

\
\foreach \point in {(6+1/2+1,1/2),(6+1/2+1,2/2),(6+1/2+1,3/2),(6+1/2+1,4/2),
  (6+1/2+1,5/2),
  (6+2/2+1,2/2),(6+2/2+1,3/2),(6+2/2+1,4/2),
  (6+3/2+1,3/2),(6+3/2+1,4/2)
}
  {
    \fill \point circle (2pt);
  }
  
\node at (3-4,2) {$n=4$};
\node at (8-4,2) {$n=5$};
\node at (8+2,2) {$n=6$};
\end{tikzpicture}
\caption{${\rm Graph}_{{\rm depth},\dim}(n)$ for
  $n=4,5,6$}\label{figure456}
\end{figure}


\begin{figure}
\begin{tikzpicture}[thick,scale=0.9, every node/.style={scale=0.9}]
\draw[thin,->] (-4.5,0) -- (-1.5,0) node[right] {${\rm depth}$};
\draw[thin,->] (-4,-0.5) -- (-4,4.5) node[above] {$\dim$};

\foreach \x [count=\xi starting from 0] in {1,2,3,4,5,6,7,8}{
        \draw (-4.1,\x/2) -- (-3.9,\x/2);
        \ifodd\xi
        \node[anchor=east] at (-4,\x/2) {$\x$};
        \fi
        \draw (1/2-4,1pt) -- (1/2-4,-1pt);
        \draw (2/2-4,1pt) -- (2/2-4,-1pt);
        \draw (3/2-4,1pt) -- (3/2-4,-1pt);
        \draw (4/2-4,1pt) -- (4/2-4,-1pt);
        \node[anchor=north] at (1/2-4,0) {$1$};
        \node[anchor=north] at (2/2-4,0) {$2$};
        \node[anchor=north] at (3/2-4,0) {$3$};
        \node[anchor=north] at (4/2-4,0) {$4$};
}

\foreach \point in {
  (1/2-4,1/2),
  (1/2-4,2/2),
  (1/2-4,3/2),
  (1/2-4,4/2),
  (1/2-4,5/2),
  (1/2-4,6/2),
  (2/2-4,2/2),
  (2/2-4,3/2),
  (2/2-4,4/2),
  (2/2-4,5/2),
  (3/2-4,3/2),
  (3/2-4,4/2),
  (3/2-4,5/2)
}{
  \fill \point circle (2pt);}

\draw[thin,->] (5.5-4,0) -- (8.5-4,0) node[right] {${\rm depth}$};
\draw[thin,->] (6-4,-0.5) -- (6-4,4.5) node[above] {$\dim$};

\foreach \x [count=\xi starting from 0] in {1,2,3,4,5,6,7,8}{
  \draw (5.9-4,\x/2) -- (6.1-4,\x/2);
    \ifodd\xi
    \fi

    \draw (6+1/2-4,1pt) -- (6+1/2-4,-1pt);
    \draw (6+2/2-4,1pt) -- (6+2/2-4,-1pt);
    \draw (6+3/2-4,1pt) -- (6+3/2-4,-1pt);
    \draw (6+4/2-4,1pt) -- (6+4/2-4,-1pt);
    \node[anchor=north] at (6+1/2-4,0) {$1$};
    \node[anchor=north] at (6+2/2-4,0) {$2$};
    \node[anchor=north] at (6+3/2-4,0) {$3$};
    \node[anchor=north] at (6+4/2-4,0) {$4$};
}

\
\foreach \point in {
  (6+1/2-4,1/2),
  (6+1/2-4,2/2),
  (6+1/2-4,3/2),
  (6+1/2-4,4/2),
  (6+1/2-4,5/2),
  (6+1/2-4,6/2),
  (6+1/2-4,7/2),
  (6+2/2-4,2/2),
  (6+2/2-4,3/2),
  (6+2/2-4,4/2),
  (6+2/2-4,5/2),
  (6+2/2-4,6/2),
  (6+3/2-4,3/2),
  (6+3/2-4,4/2),
  (6+3/2-4,5/2),
  (6+3/2-4,6/2),
  (6+4/2-4,4/2),
  (6+4/2-4,5/2),
  (6+4/2-4,6/2)
}
  {
    \fill \point circle (2pt);
}

\draw[thin,->] (5.5+1,0) -- (8.5+1.5,0) node[right] {${\rm depth}$};
\draw[thin,->] (6+1,-0.5) -- (6+1,4.5) node[above] {$\dim$};

\foreach \x [count=\xi starting from 0] in {1,2,3,4,5,6,7,8}{
    \draw (5.9+1,\x/2) -- (6.1+1,\x/2);
    \ifodd\xi
    \fi

    \draw (6+1/2+1,1pt) -- (6+1/2+1,-1pt);
    \draw (6+2/2+1,1pt) -- (6+2/2+1,-1pt);
    \draw (6+3/2+1,1pt) -- (6+3/2+1,-1pt);
    \draw (6+4/2+1,1pt) -- (6+4/2+1,-1pt);
    \draw (6+5/2+1,1pt) -- (6+5/2+1,-1pt);
    \node[anchor=north] at (6+1/2+1,0) {$1$};
    \node[anchor=north] at (6+2/2+1,0) {$2$};
    \node[anchor=north] at (6+3/2+1,0) {$3$};
    \node[anchor=north] at (6+4/2+1,0) {$4$};
    \node[anchor=north] at (6+5/2+1,0) {$5$};
}

\
\foreach \point in {
  (6+1/2+1,1/2),
  (6+1/2+1,2/2),
  (6+1/2+1,3/2),
  (6+1/2+1,4/2),
  (6+1/2+1,5/2),
  (6+1/2+1,6/2),
  (6+1/2+1,7/2),
  (6+1/2+1,8/2),
  (6+2/2+1,2/2),
  (6+2/2+1,3/2),
  (6+2/2+1,4/2),
  (6+2/2+1,5/2),
  (6+2/2+1,6/2),
  (6+2/2+1,7/2),
  (6+3/2+1,3/2),
  (6+3/2+1,4/2),
  (6+3/2+1,5/2),
  (6+3/2+1,6/2),
  (6+3/2+1,7/2),
  (6+4/2+1,4/2),
  (6+4/2+1,5/2),
  (6+4/2+1,6/2),
  (6+4/2+1,7/2),
  (6+5/2+1,6/2)}
  {
    \fill \point circle (2pt);
  }
  
\node at (3-4,1) {$n=7$};
\node at (8-4,1) {$n=8$};
\node at (8+2,1) {$n=9$};
\end{tikzpicture}
\caption{${\rm Graph}_{{\rm depth},\dim}(n)$ for
  $n=7,8,9$}\label{figure789}
\end{figure}
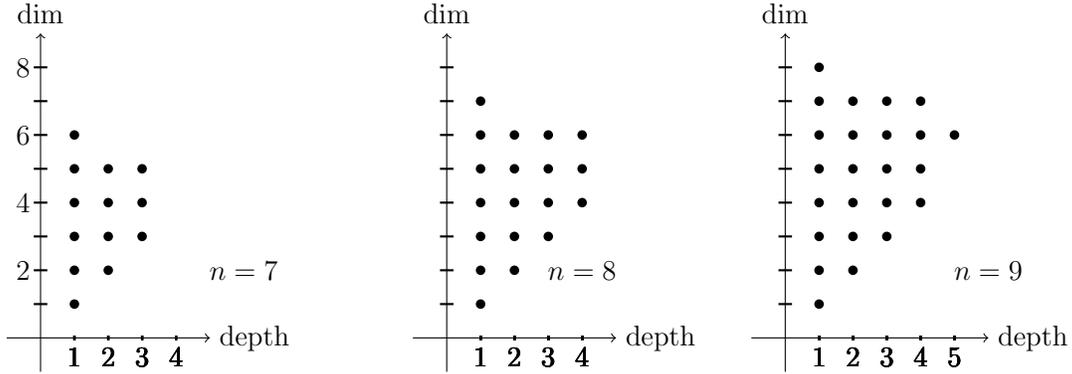

%
%
We observe  that there are big gaps
between ${\rm Graph}_{{\rm depth}, \dim}(n)$ and $C^{+} (n)$ though
$C^{-} (3) = C^{+} (3) = {\rm Graph}_{{\rm depth}, \dim} (3)$.
On the other hand, when $3 \leq n \leq 9$, the gap between
${\rm Graph}_{{\rm depth}, \dim}(n)$ and $C^{-} (n)$ is rather small.
Actually, ${\rm Graph}_{{\rm depth}, \dim}(n) = C^{-} (n)$ for $3 \leq n \leq 8$ 
and ${\rm Graph}_{{\rm depth}, \dim}(9) = C^{-} (9) \cup \{ (5,6) \}$. 
\end{example}

\par
A subset $\mathcal{M} \subseteq \mathbb{N}^{2}$ is said to be {\em convex}
if the following conditions hold: 
\begin{itemize}
    \item if $(a, b_{1}), (a, b_{2}) \in \mathcal{M}$ with $b_{1} < b_{2}$, then $(a, b) \in \mathcal{M}$ for all $b_{1} < b < b_{2}$; 
    \item if $(a_{1}, b), (a_{2}, b) \in \mathcal{M}$ with $a_{1} < a_{2}$, then $(a, b) \in \mathcal{M}$ for all $a_{1} < a < a_{2}$. 
\end{itemize}
We see that ${\rm Graph}_{{\rm depth}, \dim}(n)$ is convex for $1 \leq n \leq 9$
by Examples \ref{DD(123)} and \ref{DD(456789)}. 
This suggests the following question:
\begin{question}
  Is ${\rm Graph}_{{\rm depth}, \dim}(n) \subseteq \mathbb{N}^{2}$
  a convex subset for all $n \geq 1$ ? 
\end{question}

\smallskip


\section{Cameron--Walker graphs}\label{prelim-sec}
In this section, we focus on Cameron--Walker graphs, which are defined
below.
Let ${\rm CW}(n)$ denote the set of all Cameron--Walker graphs
with $n$ vertices and let ${\rm CW}_{{\rm depth}, \dim}(n)$ denote the set of
all possible pairs $(\depth (R/I(G)), \dim R/I(G))$ arising from
$G \in {\rm CW}(n)$: 
\[
{\rm CW}_{{\rm depth}, \dim}(n) = \{(\depth (R/I(G)), \dim R/I(G))
\, : \, G \in {\rm CW}(n) \}. 
\]
In what follows, we assume $n \geq 5$ because any Cameron--Walker graph
has at least five vertices.
The purpose of this section is to determine the set
${\rm CW}_{{\rm depth}, \dim}(n)$ for $n \geq 5$. 

\subsection{Definition of a Cameron--Walker graph}
In this subsection, we recall the definition of a Cameron--Walker graph.
As before we recall some terms from graph theory. 
A subset $\mathcal{M} \subset E(G)$ is said to be a {\em matching} of $G$ 
if $e \cap e' = \emptyset$ for any $e, e' \in \mathcal{M}$ with $e \neq e'$. 
A matching $\mathcal{M}$ of $G$ is called an {\em induced matching} of $G$ if 
for $e, e' \in \mathcal{M}$ with $e \neq e'$, there is no edge $f \in E(G)$ 
with $e \cap f \neq \emptyset$ and $e' \cap f \neq \emptyset$. 
The {\em matching number} ${\rm m}(G)$ of $G$ is the maximum cardinality 
of a matching of $G$. 
Similarly, the {\em induced matching number} ${\rm im}(G)$ of $G$
is the maximum cardinality of an induced matching of $G$. Because
an induced matching is also a matching, we always have
${\rm im}(G) \leq {\rm m}(G)$. 

By virtue of \cite[Theorem 1]{CW} together with \cite[Remark 0.1]{HHKO}, 
we have that the equality ${\rm im}(G) = {\rm m}(G)$ holds 
if and only if $G$ is one of the following graphs: 
\begin{itemize}
\item a star graph;  
\item a star triangle, i.e., a graph joining some triangles
  at one common vertex (see Figure \ref{fig:StarTriangle}); 
\item a connected finite graph consisting of a connected bipartite graph
  with vertex partition 
  $\{v_{1}, \ldots, v_{m}\} \cup \{w_{1}, \ldots, w_{p}\}$
  such that there is at least one leaf edge 
  attached to each vertex $v_{i}$ and that there may be possibly
  some pendant triangles attached to each vertex $w_{j}$;
  see Figure \ref{fig:CameronWalkerGraph}, where 
  $s_{i} \geq 1$ for all $i = 1, \ldots, m$ and 
  $t_{j} \geq 0$ for all $j = 1, \ldots, p$. 
  Note that a pendant triangle is a triangle whose two vertices
  have degree $2$ and the remaining vertex has degree more than two. 
\end{itemize} 

\begin{figure}[htbp]
  \centering
\bigskip

\begin{xy}
	\ar@{} (0,0);(80, 0)  *\cir<4pt>{} = "C2"
	\ar@{-} "C2";(60, -12)  *\cir<4pt>{} = "D2";
	\ar@{-} "C2";(64, -16)  *\cir<4pt>{} = "E2";
	\ar@{-} "D2";"E2";
	\ar@{} "C2"; (80, -5) *++!U{\cdots}
	\ar@{-} "C2";(100, -12)  *\cir<4pt>{} = "F2";
	\ar@{-} "C2";(96, -16)  *\cir<4pt>{} = "G2";
	\ar@{-} "F2";"G2"; 
\end{xy}

\bigskip

  \caption{The star triangle}
  \label{fig:StarTriangle}
\end{figure}
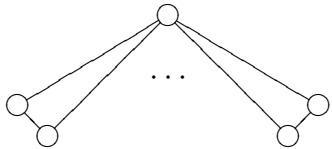

%
%
%

\begin{definition}
A finite connected simple graph $G$ is said to be a {\em Cameron--Walker graph} 
if ${\rm im}(G) = {\rm m}(G)$ and if $G$ is neither a star graph nor a star triangle. 
\end{definition}


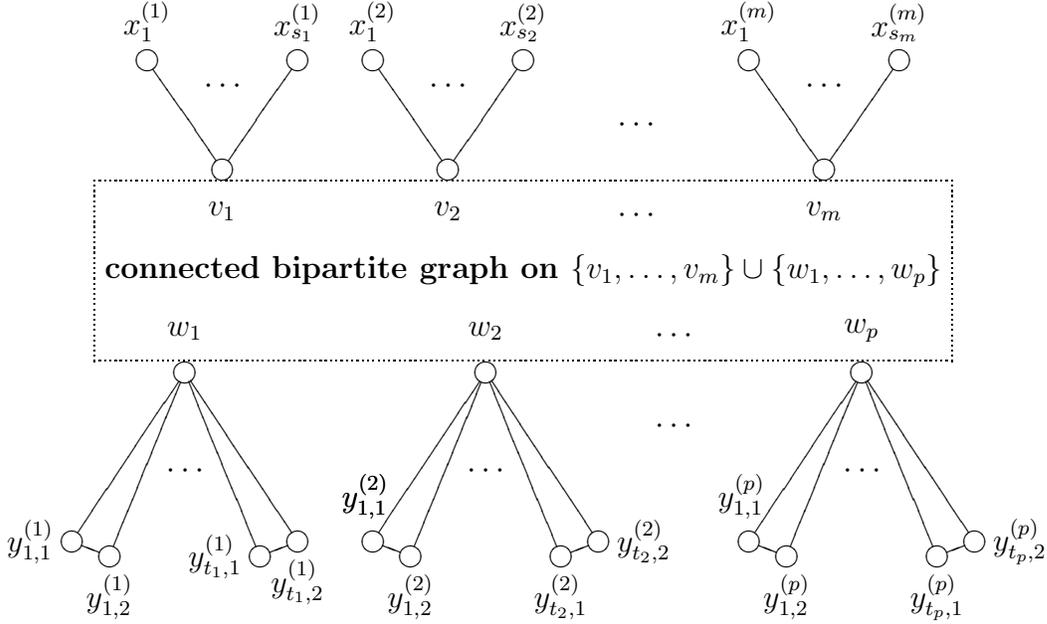
\begin{figure}[htbp]
  \centering
\begin{xy}
	\ar@{} (0,0);(18, 12)  = "A";
	\ar@{.} "A";(132, 12)  = "B";
	\ar@{.} "A";(18, -12)  = "C";
	\ar@{.} "B";(132, -12)  = "D";
	\ar@{.} "C";"D";
	\ar@{} (0,0);(75,0) *{\text{\bf{connected bipartite graph on} $\{v_{1}, \ldots, v_{m}\} \cup \{w_{1}, \ldots, w_{p}\}$}};
	\ar@{} "A";(35, 13.5)  *\cir<4pt>{} = "E1";
	\ar@{-} "E1";(25, 28) *++!D{x^{(1)}_{1}} *\cir<4pt>{};
	\ar@{-} "E1";(45, 28) *++!D{x^{(1)}_{s_{1}}} *\cir<4pt>{};
	\ar@{} (0,0); (35, 28) *++!U{\cdots}
	\ar@{} "A";(65, 13.5)  *\cir<4pt>{} = "E2";
	\ar@{-} "E2";(55, 28) *++!D{x^{(2)}_{1}} *\cir<4pt>{};
	\ar@{-} "E2";(75, 28) *++!D{x^{(2)}_{s_{2}}} *\cir<4pt>{};
	\ar@{} (0,0); (65, 28) *++!U{\cdots}
	\ar@{} "A";(115, 13.5)  *\cir<4pt>{} = "E";
	\ar@{-} "E";(105, 28) *++!D{x^{(m)}_{1}} *\cir<4pt>{};
	\ar@{-} "E";(125, 28) *++!D{x^{(m)}_{s_{m}}} *\cir<4pt>{};
	\ar@{} (0,0); (115, 28) *++!U{\cdots}
	\ar@{} (0,0);(35,8) *{\text{$v_{1}$}};
	\ar@{} (0,0);(65,8) *{\text{$v_{2}$}};
	\ar@{} (0,0);(90,4) *++!D{\cdots};
        \ar@{} (0,0);(115,8) *{\text{$v_{m}$}};
	\ar@{} (0,0);(90,16) *++!D{\cdots};
	\ar@{} "E1";(30, -13.5)  *\cir<4pt>{} = "F10";
	\ar@{-} "F10";(15, -36) *++!R{y^{(1)}_{1,1}} *\cir<4pt>{} = "F11";
	\ar@{-} "F10";(20, -38) *++!U{y^{(1)}_{1,2}} *\cir<4pt>{} = "F12";
	\ar@{-} "F11";"F12";
	\ar@{-} "F10";(40, -38) *++!R{y^{(1)}_{t_{1},1}} *\cir<4pt>{} = "Ft1";
	\ar@{-} "F10";(45, -36) *++!U{y^{(1)}_{t_{1},2}} *\cir<4pt>{} = "Ft2";
	\ar@{-} "Ft1";"Ft2";
	\ar@{} "E1";(70, -13.5)  *\cir<4pt>{} = "F20";
	\ar@{-} "F20";(55, -36) *\cir<4pt>{} = "F21";
	\ar@{} (0,0);(54,-30) *{\text{$y^{(2)}_{1,1}$}};
	\ar@{-} "F20";(60, -38) *++!U{y^{(2)}_{1,2}} *\cir<4pt>{} = "F22";
	\ar@{-} "F21";"F22";
	\ar@{-} "F20";(80, -38) *++!U{y^{(2)}_{t_{2},1}} *\cir<4pt>{} = "Ftt1";
	\ar@{-} "F20";(85, -36) *++!L{y^{(2)}_{t_{2},2}} *\cir<4pt>{} = "Ftt2";
	\ar@{-} "Ftt1";"Ftt2";
	\ar@{} "E1";(120, -13.5)  *\cir<4pt>{} = "Fn0";
	\ar@{-} "Fn0";(105, -36) *\cir<4pt>{} = "Fn1";
	\ar@{} (0,0);(104,-30) *{\text{$y^{(p)}_{1,1}$}};
	\ar@{-} "Fn0";(110, -38) *++!U{y^{(p)}_{1,2}} *\cir<4pt>{} = "Fn2";
	\ar@{-} "Fn1";"Fn2";
	\ar@{-} "Fn0";(130, -38) *++!U{y^{(p)}_{t_{p},1}} *\cir<4pt>{} = "Fnn1";
	\ar@{} (0,0);(54,-30) *{\text{$y^{(2)}_{1,1}$}};
	\ar@{-} "Fn0";(135, -36) *++!L{y^{(p)}_{t_{p},2}} *\cir<4pt>{} = "Fnn2";
	\ar@{-} "Fnn1";"Fnn2";
	\ar@{} (0,0);(30,-8) *{\text{$w_{1}$}};
	\ar@{} (0,0);(70,-8) *{\text{$w_{2}$}};
	\ar@{} (0,0);(95,-12) *++!D{\cdots};
	\ar@{} (0,0);(120,-8) *{\text{$w_{p}$}};
	\ar@{} (0,0);(95,-24) *++!D{\cdots};
	\ar@{} (0,0);(30,-30) *++!D{\cdots};
	\ar@{} (0,0);(70,-30) *++!D{\cdots};
	\ar@{} (0,0);(120,-30) *++!D{\cdots};
\end{xy}

\bigskip

  \caption{A Cameron--Walker graph}
  \label{fig:CameronWalkerGraph}
\end{figure}

For a Cameron--Walker graph $G$ with notation as in
Figure \ref{fig:CameronWalkerGraph},
we denote by $G_{\mathrm{bip}}$ the bipartite part of $G$, 
namely, the
induced subgraph of
$G$ on $\{ v_1, \ldots, v_m \} \cup \{ w_1, \ldots, w_p \}$. 

\par
We collect some known formulas from \cite{HHKO,HKMT,HKMVT}
for homological invariants of Cameron--Walker graphs
for later use.
Here we set
\begin{displaymath}
  i(G) = \min \{ |A| \; : \;
  \text{$A$ is an independent set of $G$ with $A \cup N_{G}(A) = V(G)$} \},
\end{displaymath}
where $N_{G}(A) = \bigcup_{v \in A}N_{G}(v) \setminus A$
as in \cite[Section 4]{DaoSchweig}.
\begin{theorem}[{\cite{HHKO, HKMT,HKMVT}}] 
  \label{CWformulas}
  Let $G$ be a Cameron--Walker graph with notation
  as in Figure \ref{fig:CameronWalkerGraph}. 
  Then 
  \begin{enumerate}
  \item[$(i)$]
    $\displaystyle |V(G)| = m + p + \sum_{i=1}^{m} s_{i} + 2 \sum_{j=1}^{p} t_{j}$. 
  \item[$(ii)$] $($\cite[Proposition 1.3]{HKMT}$)$
    \begin{displaymath}
      \dim R/I(G) = \deg h(R/I(G))
      = \sum_{i=1}^m s_i + \sum_{j=1}^{p} t_j + |\{j:t_{j} = 0\}|.
    \end{displaymath}
  \item[$(iii)$] $($\cite[Corollary 3.7]{HHKO}$)$ $\depth (R/I(G)) = i(G)$.
  \item[$(iv)$] $($\cite[Lemma 2.1]{HKMT}$)$ The inequalities
    \begin{displaymath}
      m + |\{j \; : \; t_{j} > 0\}| \leq \depth (R/I(G))
      \leq \min \left\{ p + \sum_{i=1}^m s_i, \, m + \sum_{j = 1}^{p} t_{j} \right\} 
    \end{displaymath}
    hold. Moreover, if the bipartite part of $G$
    is the complete bipartite graph, then 
    \begin{displaymath}
      \depth (R/I(G))
      = \min\left\{ p + \sum_{i = 1}^{m} s_{i}, m + \sum_{j = 1}^{p} t_{j} \right\}.
    \end{displaymath}
  \item[$(v)$] $($\cite[Lemma 4.2]{HKMVT}$)$ $\displaystyle \reg (R/I(G)) = m + \sum_{j=1}^p t_j$.
    In particular, $\depth (R/I(G)) \leq \reg (R/I(G))$. 
    \end{enumerate}
\end{theorem}

%

The formula $\depth (R/I(G)) = i(G)$ (Theorem \ref{CWformulas} (iii)) 
is a characterization of $\depth (R/I(G))$ in terms of
the combinatorics of a Cameron--Walker graph $G$. 
However we can describe $i(G)$, and thus $\depth (R/I(G))$, 
according to the structure of a Cameron--Walker graph $G$. 
We will use this characterization in Section \ref{sec:MainResult}
for the proof of Theorem \ref{maintheorem3}. 
\begin{theorem}
  \label{CWdepthV}
  Let $G$ be a Cameron--Walker graph with notation
  as in Figure \ref{fig:CameronWalkerGraph}. 
  For a subset $V \subset \{ v_{1}, \ldots, v_{m} \}$, we set
  \begin{displaymath}
    f(V) = \sum_{v_{i} \in V} s_{i} + m - |V|
      + \sum_{N_{G_{\mathrm{bip}}} (w_{j}) \not\subset V} t_{j}
      + \left|\left\{j : N_{G_{\mathrm{bip}}} (w_{j}) \subset V \right\}\right|. 
  \end{displaymath}
  Then $\displaystyle i(G) = \min_{V \subset \{ v_{1}, \ldots, v_{m} \} }
  \left\{ f(V) \right\}$.
  In particular,  
  \begin{displaymath}
    \depth (R/I(G))
      = \min_{V \subset \{ v_{1}, \ldots, v_{m} \} } \left\{ f(V)  \right\}. 
  \end{displaymath}
\end{theorem}
\begin{proof}
%
  First, let $A$ be an independent set with $A \cup N_{G}(A) = V(G)$. 
  We set $A_{v} = A \cap \{v_{1}, \ldots, v_{m}\}$
  and $A_{w} = A \cap \{w_{1}, \ldots, w_{p}\}$.
  Note that
\begin{itemize}
\item if $v_{i} \not\in A_{v}$,
  then $x_{k}^{(i)} \in A$ for all $1 \leq k \leq s_{i}$; 
\item if $w_{j} \not\in A_{w}$ and $t_{j} > 0$,
  then $y_{\ell, 1}^{(j)} \in A$ or $y_{\ell, 2}^{(j)} \in A$
  for all $1 \leq \ell \leq t_{j}$; 
\item if $v_{i} \in A_{v}$, then $N_{G_{\mathrm{bip}}}(v_{i}) \cap A_{w} = \emptyset$; 
\item if $w_{j} \in A_{w}$, then $N_{G_{\mathrm{bip}}}(w_{j}) \cap A_{v} = \emptyset$; and
\item if $w_{j} \not\in A_{w}$ and $N_{G_{\mathrm{bip}}}(w_{j}) \cap A_{v} = \emptyset$,
  then $t_{j} > 0$. 
\end{itemize}
Hence we have 
  \[
  |A| = \sum_{v_{i} \not\in A_{v}}s_{i} + |A_{v}|
  + \sum_{w_{j} \not\in A_{w}} t_{j} + |A_{w}| 
  \]
and
\begin{eqnarray*}
& & f(\{v_{1}, \ldots, v_{m}\} \setminus A_{v}) \\ 
  &=& \sum_{v_{i} \not\in A_{v}} s_{i} + m - (m - |A_{v}|)
  + \sum_{N_{G_{\mathrm{bip}}} (w_{j}) \cap A_{v} \neq \emptyset} t_{j}
  + \left|\left\{j : N_{G_{\mathrm{bip}}} (w_{j}) \cap A_{v}
  = \emptyset \right\}\right| \\ 
  &=& \sum_{v_{i} \not\in A_{v}} s_{i} + |A_{v}|
  + \sum_{N_{G_{\mathrm{bip}}} (w_{j}) \cap A_{v} \neq \emptyset} t_{j}
  + \left|\left\{j : N_{G_{\mathrm{bip}}} (w_{j}) \cap A_{v}
  = \emptyset \right\}\right| \\
  &=& \sum_{v_{i} \not\in A_{v}} s_{i} + |A_{v}|
  + \sum_{N_{G_{\mathrm{bip}}} (w_{j}) \cap A_{v} \neq \emptyset \atop w_{j} \not\in A_{w}} t_{j}
  +  \left|\left\{j : N_{G_{\mathrm{bip}}} (w_{j}) \cap A_{v} = \emptyset,
  w_{j} \not\in A_{w} \right\}\right| + |A_{w}| \\
  &\leq& \sum_{v_{i} \not\in A_{v}} s_{i} + |A_{v}|
  + \sum_{N_{G_{\mathrm{bip}}} (w_{j}) \cap A_{v} \neq \emptyset \atop w_{j} \not\in A_{w}} t_{j}
  + \sum_{N_{G_{\mathrm{bip}}} (w_{j}) \cap A_{v} = \emptyset \atop w_{j} \not\in A_{w}} t_{j}
  + |A_{w}| \\
  &=& \sum_{v_{i} \not\in A_{v}} s_{i} + |A_{v}| + \sum_{w_{j} \not\in A_{w}} t_{j}
  + |A_{w}| = |A|. 
\end{eqnarray*}
Thus it follows that $\displaystyle \min_{V \subset \{v_{1}, \ldots, v_{m}\}}
\left\{ f(V) \right\} \leq i(G)$. 

\par
Next, take a subset $V \subset \{v_{1}, \ldots, v_{m}\}$.
Then we consider the following subset $A(V) \subset V(G)$: 
\begin{displaymath}
  \begin{aligned}
    A(V) &= \left( \bigcup_{v_{i} \in V} \left\{ x_{1}^{(i)}, \ldots, x_{s_{i}}^{(i)}
    \right\} \right) \cup \left( \{v_{1}, \ldots, v_{m} \} \setminus V \right) \\
    &\cup \left( \bigcup_{N_{G_{\mathrm{bip}}} (w_{j}) \not\subset V \atop t_{j} > 0}
    \left\{ y_{1, 1}^{(j)}, \ldots, y_{t_{j}, 1}^{(j)} \right\} \right) 
    \cup \left\{w_{j} : N_{G_{\mathrm{bip}}} (w_{j}) \subset V \right\}. 
  \end{aligned}
\end{displaymath}
It is easy to see that $A(V)$ is an independent set
with $A \cup N_{G}(A) = V(G)$ and $|A(V)| = f(V)$. 
Hence one has $\displaystyle \min_{V \subset \{v_{1}, \ldots, v_{m}\}}
\left\{ f(V) \right\} \geq i(G)$. 
Therefore we have the desired conclusion. 
\end{proof}
\begin{remark}
  In \cite[Lemma 2.1]{HKMT}, the authors derived Theorem \ref{CWformulas} (iv)
  by investigating $i(G)$.
  We can obtain the same result using
  $\min_{V \subset \{ v_{1}, \ldots, v_{m} \} }
  \left\{ f(V) \right\}$. The proof is as follows: 

  \par
  Let $G$ be a Cameron--Walker graph with notation as in Figure
  \ref{fig:CameronWalkerGraph}.
  Since $s_{i} \geq 1$, for any subset $V \subset \{ v_1, \ldots, v_m \}$,
  one has
    \begin{eqnarray*}
      f(V) &=& \sum_{v_{i} \in V} s_{i} + m - |V|
      + \sum_{N_{G_{\mathrm{bip}}} (w_{j}) \not\subset V} t_{j}
      + \left|\left\{j : N_{G_{\mathrm{bip}}} (w_{j}) \subset V \right\}\right| \\
      &=& m + \left( \sum_{v_{i} \in V} s_{i} - |V| \right)
      + \sum_{N_{G_{\mathrm{bip}}} (w_{j}) \not\subset V \atop t_{j} > 0} t_{j}
      + \left|\left\{j : N_{G_{\mathrm{bip}}} (w_{j}) \subset V , t_{j} > 0 \right\}\right| \\
    &+& \left|\left\{j : N_{G_{\mathrm{bip}}} (w_{j}) \subset V, t_{j} = 0 \right\}\right| \\
      &\geq& m + \left|\left\{j : N_{G_{\mathrm{bip}}} (w_{j}) \not\subset V , t_{j} > 0 \right\}\right|
      + \left|\left\{j : N_{G_{\mathrm{bip}}} (w_{j}) \subset V , t_{j} > 0 \right\}\right| \\
    &=& m + |\{j:t_{j} > 0\}|. 
    \end{eqnarray*}
    Hence $m + |\{j:t_{j} > 0\}| \leq {\rm depth}(R/I(G))$.

    \par
    On the other hand,
    $f(\{ v_1, \ldots, v_m \}) = p + \sum_{i=1}^m s_i$ and
    $f(\emptyset) = m + \sum_{j=1}^p t_j$ imply the upper bound for
    $\depth (R/I(G))$. 

   \par
   For the latter assertion, assume that the bipartite part of
   $G$ is the complete bipartite graph. 
   If $V = \{v_{1}, \ldots, v_{m}\}$,
   then $\displaystyle f(V) = \sum_{i = 1}^{m}s_{i} + p$.
   If $V \subsetneq \{v_{1}, \ldots, v_{m}\}$, then
      \[
      f(V) = \sum_{v_{i} \in V} s_{i} + m - |V|
      + \sum_{j = 1}^{p} t_{j} \geq m + \sum_{j = 1}^{p} t_{j} = f(\emptyset)
      \]
      since $s_{i} \geq 1$ for all $1 \leq i \leq m$.
      Therefore we have the desired conclusion. 
\end{remark}

\par
In order to investigate ${\rm CW}_{{\rm depth}, \dim}(n)$,
we require more relations among $|V(G)|$, $\depth (R/I(G))$,
and $\dim R/I(G)$.  
\begin{lemma}\label{CWDDLemma}
  Let $G$ be a Cameron--Walker graph with notation
  as in Figure \ref{fig:CameronWalkerGraph}.  Then
  \begin{enumerate}
  \item[$(i)$] $2 \leq {\rm depth}(R/I(G))
    \leq \left\lfloor \frac{|V(G)| - 1}{2} \right\rfloor$. 
  \item[$(ii)$] ${\rm depth}(R/I(G)) + \dim R/I(G) \leq |V(G)|$. 
  \item[$(iii)$] $|V(G)| < {\rm depth}(R/I(G)) + 2\dim R/I(G)$. 
  \item[$(iv)$] If $R/I(G)$ is Cohen-Macaulay,
    then $|V(G)| = 2m + 3p$ and $\dim R/I(G) = {\rm depth}(R/I(G)) = m + p$.  
  \end{enumerate}
\end{lemma}
\begin{proof}
  \begin{enumerate}
  \item[$(i)$] It follows from Theorem \ref{CWformulas} (v),
    \cite[Proposition 2.8]{HKMT} and \cite[Theorem 5.1]{HKMVT}. 
  \item[$(ii)$] It follows from Theorem \ref{CWformulas} (ii), (v)
    and \cite[Theorem 13]{HMVT}. 
  \item[$(iii)$] By virtue of Theorem \ref{CWformulas} (i), (ii), (iv),
    we have 
    \begin{eqnarray*}
    |V(G)| - \dim R/I(G) &=& m + p + \sum_{j = 1}^{p} t_{j} - |\{j:t_{j} = 0\}| \\
    &=& m + |\{j:t_{j} > 0\}| + \sum_{j = 1}^{p} t_{j} \\
    &<& {\rm depth}(R/I(G)) + \dim R/I(G). 
    \end{eqnarray*}
    Hence one has $|V(G)| < {\rm depth}(R/I(G)) + 2\dim R/I(G)$.
  \item[$(iv)$] Assume that $R/I(G)$ is Cohen-Macaulay. 
    Then $s_{i} = t_{j} = 1$ for all $1 \leq i \leq m$
    and for all $1 \leq j \leq p$ by virtue of \cite[Theorem 1.3]{HHKO}. 
    Thus we have $|V(G)| = 2m + 3p$ and
    ${\rm depth}(R/I(G)) = \dim R/I(G) =  m + p$
    from Theorem \ref{CWformulas} (i), (ii).  
  \end{enumerate}
\end{proof}

\begin{remark}
  The inequalities (i), (ii) and (iii) of Lemma \ref{CWDDLemma} 
  do not hold for non-Cameron--Walker graphs in general. 
  We give some examples. 
  \begin{enumerate}
  \item[(i)] Let $G = G(2; 1, 1)$ be the graph
    in Construction \ref{const:G(m;s)}.
    Then ${\rm depth}(R/I(G)) = 2 > 1
    = \left\lfloor \frac{|V(G)| - 1}{2} \right\rfloor$
    by virtue of Proposition \ref{G(m)}. 
  \item[(ii)] Let $G = G(2; 2, 2)$ be the graph in Construction
    \ref{const:G(m;s)}. Then ${\rm depth}(R/I(G)) + \dim R/I(G) = 7 > 6
    = |V(G)|$ virtue of Proposition \ref{G(m)}. 
  \item[(iii)] Let $G = K_{4}$ be the complete graph with $4$ vertices.
    Then $|V(G)| = 4 > 3 = {\rm depth}(R/I(G)) + 2\dim R/I(G)$. 
  \end{enumerate}
\end{remark}

\subsection{Special families of Cameron--Walker graphs}
In this subsection,
we introduce some special families of Cameron--Walker graphs 
and compute invariants of its edge ideal for later use. 
The first special family is constructed as follows:
\begin{construction}
  \label{const:specialCW1}
 Let $m, p, t \geq 1$ be integers.  
 Let $G = G^{(1)}_{m,p,t}$ be the Cameron--Walker
  graph whose bipartite part is the complete bipartite graph $K_{m, p}$ with 
  $s_1 = \cdots = s_m = 1$, $t_{1} = \cdots = t_{p - 1} = 1$ and $t_{p} = t$; see Figure \ref{fig:specialCW1}. 
  
  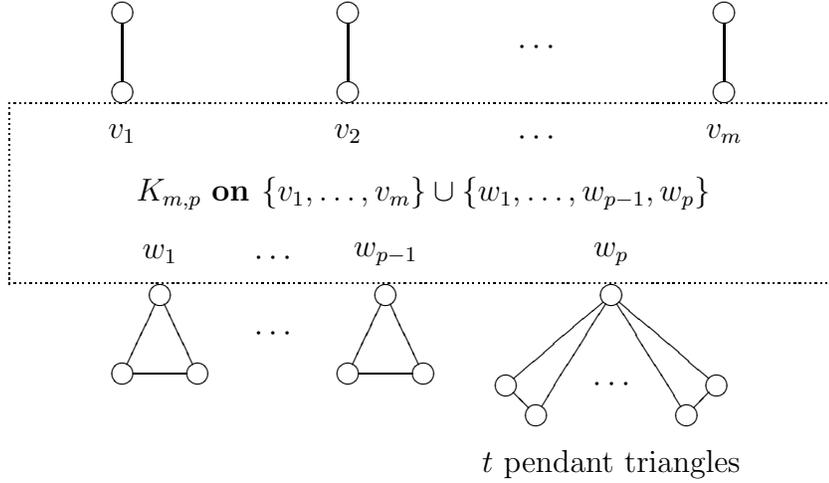
\begin{figure}[htbp]
  \centering
\begin{xy}
	\ar@{} (0,0);(20, 12)  = "A";
	\ar@{.} "A";(130, 12)  = "B";
	\ar@{.} "A";(20, -12)  = "C";
	\ar@{.} "B";(130, -12)  = "D";
	\ar@{.} "C";"D";
	\ar@{} (0,0);(75,0) *{\text{$K_{m, p}$ \bf{on} $\{v_{1}, \ldots, v_{m}\} \cup \{w_{1}, \ldots, w_{p - 1}, w_{p}\}$}};
	\ar@{} "A";(35, 13.5)  *\cir<4pt>{} = "E1";
	\ar@{-} "E1";(35, 24)  *\cir<4pt>{};
	\ar@{} "A";(65, 13.5)  *\cir<4pt>{} = "E2";
	\ar@{-} "E2";(65, 24)  *\cir<4pt>{};
	\ar@{} "A";(115, 13.5)  *\cir<4pt>{} = "E";
	\ar@{-} "E";(115, 24)  *\cir<4pt>{};
	\ar@{} (0,0);(35,8) *{\text{$v_{1}$}};
	\ar@{} (0,0);(65,8) *{\text{$v_{2}$}};
	\ar@{} (0,0);(90,4) *++!D{\cdots};
    \ar@{} (0,0);(115,8) *{\text{$v_{m}$}};
	\ar@{} (0,0);(90,16) *++!D{\cdots};
	\ar@{} "E1";(40, -13.5)  *\cir<4pt>{} = "Fc+10"; 
	\ar@{-} "Fc+10";(35, -24) *\cir<4pt>{} = "Fc+11";
	\ar@{-} "Fc+10";(45, -24) *\cir<4pt>{} = "Fc+12";
	\ar@{-} "Fc+11";"Fc+12";
	\ar@{} "E1";(70, -13.5)  *\cir<4pt>{} = "Fn0";
	\ar@{-} "Fn0";(65, -24) *\cir<4pt>{} = "Fn1";
	\ar@{-} "Fn0";(75, -24) *\cir<4pt>{} = "Fn2";
	\ar@{-} "Fn1";"Fn2";
	\ar@{} (0,0);(40,-8) *{\text{$w_{1}$}};
	\ar@{} (0,0);(55,-22) *++!D{\cdots};
	\ar@{} (0,0);(70,-8) *{\text{$w_{p - 1}$}};
	\ar@{} (0,0);(55,-12) *++!D{\cdots};
	\ar@{} (0,0);(100,-8) *{\text{$w_{p}$}};
	\ar@{} (0,0);(100,-13.5) *\cir<4pt>{} = "F";
	\ar@{-} "F";(86, -25.5) *\cir<4pt>{} = "T1";
    \ar@{-} "F";(90, -29.5) *\cir<4pt>{} = "T2";
    \ar@{-} "T1";"T2";
    \ar@{-} "F";(114, -25.5) *\cir<4pt>{} = "T3";
    \ar@{-} "F";(110, -29.5) *\cir<4pt>{} = "T4";
    \ar@{-} "T3";"T4";
    \ar@{} (0,0); (100, -22) *++!U{\cdots};
    \ar@{} (0,0);(100, -36) *{t \ \text{pendant triangles}};
\end{xy}

  \caption{The Cameron--Walker graph $G^{(1)}_{m,p,t}$}
  \label{fig:specialCW1}
\end{figure}
\end{construction}

We can compute the homological invariants of the edge ideal of $G^{(1)}_{m,p,t}$
by Theorem \ref{CWformulas}. 
\begin{lemma}\label{ddspecialCW1}
  Let $G = G^{(1)}_{m,p,t}$ be the Cameron--Walker
  graph in Construction \ref{const:specialCW1}.
  Then $|V(G)| = 2m + 3p + 2t - 2$,
  $\dim R/I(G) = \deg h(R/I(G)) = \reg (R/I(G)) = m + p + t - 1$,
  and ${\rm depth}(R/I(G)) = m + p$. 
\end{lemma}

The second special family is as follows:
\begin{construction}
\label{const:specialCW2}
  Let $m \geq 2$, $s \geq 1$ and $t \geq 1$ be integers. 
  Let $G = G^{(2)}_{m,s,t}$ be the Cameron--Walker graph with $p=2$, 
  $s_1 = \cdots = s_{m - 1} = 1$, $s_{m} = s$, $t_1 = t$, and 
  $t_2 =0$ such that 
  \begin{displaymath}
    E(G_{\rm bip}) = \big\{ \{ v_1, w_1 \}, 
      \; \{ v_1, w_2 \}, \{ v_2, w_2 \}, \ldots, \{ v_{m}, w_2 \} \big\}; 
  \end{displaymath}
  see Figure \ref{fig:specialCW2}. 
\end{construction}  
 
  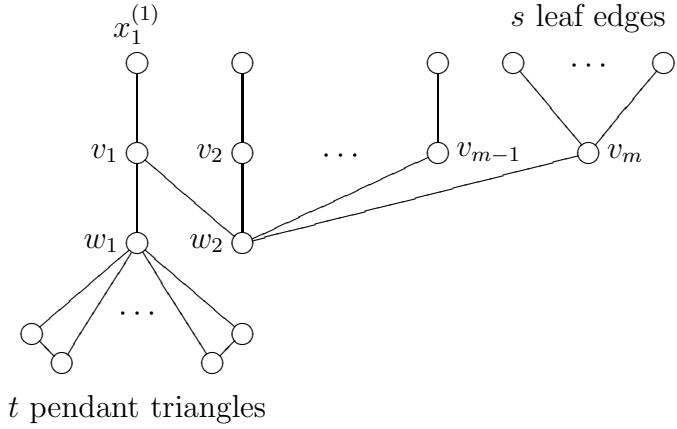
\begin{figure}[htbp]
  \centering
  \begin{xy}
    \ar@{} (0,0);(50, 6) *++!R{v_{1}} *\cir<4pt>{} = "A1";
    \ar@{-} "A1";(50, -6) *++!R{w_{1}} *\cir<4pt>{} = "A2";
    \ar@{} (0,0);(64, 6) *++!R{v_{2}} *\cir<4pt>{} = "B1";
    \ar@{-} "B1";(64, -6) *++!R{w_{2}} *\cir<4pt>{} = "B2";
    \ar@{-} "A1";"B2";
    \ar@{-} "B2";(90, 6) *++!L{v_{m-1}} *\cir<4pt>{} = "C";
    \ar@{-} "B2";(110, 6) *++!L{v_{m}} *\cir<4pt>{} = "D";
    \ar@{} (0,0); (77, 9) *++!U{\cdots};
    \ar@{-} "A1";(50, 18) *++!D{x^{(1)}_{1}} *\cir<4pt>{};
    \ar@{-} "B1";(64, 18) *\cir<4pt>{};
    \ar@{-} "C";(90, 18) *\cir<4pt>{};
    \ar@{-} "D";(100, 18) *\cir<4pt>{};
    \ar@{-} "D";(120, 18) *\cir<4pt>{};
    \ar@{} (0,0); (110, 21) *++!U{\cdots};
    \ar@{-} "A2";(36, -18) *\cir<4pt>{} = "T1";
    \ar@{-} "A2";(40, -22) *\cir<4pt>{} = "T2";
    \ar@{-} "T1";"T2";
    \ar@{-} "A2";(64, -18) *\cir<4pt>{} = "T3";
    \ar@{-} "A2";(60, -22) *\cir<4pt>{} = "T4";
    \ar@{-} "T3";"T4";
    \ar@{} (0,0); (50, -12) *++!U{\cdots};
    \ar@{} (0,0);(110, 24) *{s \ \text{leaf edges}};
    \ar@{} (0,0);(50, -28) *{t \ \text{pendant triangles}};
  \end{xy}

  \bigskip

  \caption{The Cameron--Walker graph $G^{(2)}_{m,s,t}$}
  \label{fig:specialCW2}
  \end{figure}

We compute the homological invariants of the edge ideal of $G^{(2)}_{m,s,t}$. 
\begin{lemma}\label{ddspecialCW2}
  Let $G = G^{(2)}_{m,s,t}$ be the Cameron--Walker graph
  in Construction \ref{const:specialCW2}.
  Then $|V(G)| = 2m + s + 2t + 1$, $\dim R/I(G) = \deg h(R/I(G)) = m + s + t$,
  ${\rm depth} (R/I(G)) = m + 1$ and $\reg (R/I(G)) = m+t$. 
\end{lemma}
\begin{proof}  
  We only check the depth. The other invariants are easily follows from
  Theorem \ref{CWformulas}.

  \par
  Let $A = \{ v_2, \ldots, v_{m}, w_{1}, x_1^{(1)}\}$.  
  Then $A$ is an independent set of $V(G)$ with $A \cup N_G (A) = V(G)$.
  Hence one has ${\rm depth}(R/I(G)) = i(G) \leq |A| = m + 1$ by virtue of
  Theorem \ref{CWformulas} (iii). 
  Moreover, Theorem \ref{CWformulas} (iv) says that
  ${\rm depth}(R/I(G)) \geq m + 1$. 
  Thus we have ${\rm depth}(R/I(G)) = m + 1$. 
\end{proof}


\subsection{Cameron--Walker graphs of small depth}
In this subsection we will deal with all the Cameron--Walker graphs
of depth $\leq 2$. Such graphs are classified in \cite[Proposition 2.8]{HKMT}.
We will compute homological invariants of them and determine the
elements in ${\rm CW}_{{\rm depth}, \dim}(n)$ of the form $(2,b)$. 
As an application, we prove that ${\rm CW}_{{\rm depth}, \dim}(n)$ is not
convex for all odd integers $n \geq 9$.

\par
We first recall a classification for the Cameron--Walker graphs
of depth $\leq 2$. 
\begin{lemma} $($\cite[Proposition 2.8]{HKMT}$)$\
  \label{depth2CW}
  There is no Cameron--Walker graph $G$ on $n$ vertices with
  ${\rm depth}(R/I(G)) = 1$.

  \par
  Let $G$ be a Cameron--Walker graph with $\depth (R/I(G)) =2$.
  Then $G$ belongs to one of the following
  families of Cameron--Walker graphs
  with the notation as in Figure \ref{fig:CameronWalkerGraph}:
  \begin{enumerate}
  \item[(e1)] $m = 2$ and $t_j = 0$ for all $1 \leq j \leq p$;
  \item[(e2)] $m=p=1$ and $t_1 = 1$;
  \item[(e3)] $m=p=1$, $t_1 \geq 2$ and $s_1 = 1$.
  \end{enumerate}
\end{lemma}

We compute the homological invariants of Cameron--Walker graphs $G$ 
with $\depth (R/I(G)) = 2$. 
\begin{lemma}
  \label{depth2Invariants}
  Let $G$ be a Cameron--Walker graph with $\depth (R/I(G)) = 2$.
  With the notation as in
  Figure \ref{fig:CameronWalkerGraph}, the invariants of $R/I(G)$ are:
  \begin{enumerate}
  \item When $G$ is type (e1) in Lemma \ref{depth2CW},
    $|V(G)| = s_1 + s_2 + p + 2$,
    $\dim R/I(G) = \deg h(R/I(G)) = s_1 + s_2 + p$, and
    $\reg (R/I(G)) = 2$. 
  \item When $G$ is type (e2) in Lemma \ref{depth2CW},
    $|V(G)| = s_1 + 4$,
    $\dim R/I(G) = \deg h(R/I(G)) = s_1 + 1$, and 
    $\reg (R/I(G)) = 2$. 
  \item When $G$ is type (e3) in Lemma \ref{depth2CW},
    $|V(G)| = 2 t_1 + 3$ and 
    $\dim R/I(G) = \deg h(R/I(G)) = \reg (R/I(G)) = t_1 + 1$.
  \end{enumerate}
\end{lemma}
\begin{proof}
  We can derive these results by Theorem \ref{CWformulas}.
\end{proof}

By virtue of Lemmas \ref{depth2CW} and \ref{depth2Invariants},
we have the following
\begin{proposition}\label{depth2}
  Let $n \geq 5$ be an integer.
  Assume that $n$ is even (resp.\  odd). 
  Then $(2, b) \in {\rm CW}_{{\rm depth}, \dim}(n)$ if and only if
  $b = n - 2$ or $b = n - 3$
  (resp.\ $b = n - 2$, $b = n - 3$ or $b = (n - 1)/2$). 
\end{proposition}
\begin{proof}
  If $(2,b) \in {\rm CW}_{{\rm depth}, \dim}(n)$, then it is easy to see that
  $b = n-2$ or $b=n-3$ when $n$ is even;  
  $b = n-2$, $b=n-3$ or $b= (n-1)/2$ when $n$ is odd
  by Lemma \ref{depth2Invariants}.

  \par
  We check the other implication.

  \par
  First assume that $b= n-2$. Then the Cameron--Walker graph $G$ of type (e1)
  in Lemma \ref{depth2CW} with $s_1 = s_2 = 1$ and $p = b-2 (\geq 1)$ satisfies
  $|V(G)|=n$ and $(\depth (R/I(G)), \dim R/I(G)) = (2, b)$
  by Lemma \ref{depth2Invariants}. 

  \par
  Next assume that $b= n-3$. Then the Cameron--Walker graph $G$ of type (e2)
  in Lemma \ref{depth2CW} with $s_1 = b-1 (\geq 1)$ satisfies $|V(G)|=n$
  and $(\depth (R/I(G)), \dim R/I(G)) = (2, b)$
  by Lemma \ref{depth2Invariants}. 

  \par
  Lastly, assume that $n$ is odd and $b = (n-1)/2$.
  When $n\geq 7$, the Cameron--Walker graph $G$ of type (e3) 
  in Lemma \ref{depth2CW} with $t_1 = b-1 (\geq 2)$ satisfies
  $|V(G)|=n$ and $(\depth (R/I(G)), \dim R/I(G)) = (2, b)$
  by Lemma \ref{depth2Invariants}.
  When $n=5$ (then $b = 2$), the Cameron--Walker graph $G$ of type (e2) 
  in Lemma \ref{depth2CW} with $s_1 = 1$ satisfies
  $|V(G)|=n$ and $(\depth (R/I(G)), \dim R/I(G)) = (2, b) = (2,2)$
  by Lemma \ref{depth2Invariants}.
\end{proof}

As a corollary, we have 
\begin{corollary}
  If $n \geq 9$ is an odd integer,
  then ${\rm CW}_{{\rm depth}, \dim}(n)$ is not convex. 
\end{corollary}
\begin{proof}
  When $n$ is odd, $(n-3) - (n-1)/2 = (n-5)/2$.
  Hence if $n \geq 9$, there is a gap between $(2,n-3)$ and $(2,(n-1)/2)$.
\end{proof}

%

\subsection{Lattice points of ${\rm CW}_{{\rm depth}, \dim}(n)$}
Now we determine the set ${\rm CW}_{{\rm depth}, \dim} (n)$ for any integer
$n \geq 5$.
Note that by Lemma \ref{CWDDLemma},
if $(a,b) \in {\rm CW}_{{\rm depth}, \dim} (n)$, then inequalities 
\begin{displaymath}
  a \leq b, \quad  2 \leq a \leq \left\lfloor \frac{n-1}{2} \right\rfloor,
  \quad  a+b \leq n, \quad  n < a+2b
\end{displaymath}
hold. 
The following theorem is the main result of this section. 
\begin{theorem}\label{CWdd}
    Let $n \geq 5$ be an integer. Then 
    \begin{displaymath}
      \begin{aligned}
        {\rm CW}_{{\rm depth}, \dim} (n)
        &= {\rm CW}_{2, \dim} (n) \cup \left\{ (b, b) \in \mathbb{N}^{2}
             \  \middle| \  \frac{n}{3} < b < \frac{n}{2} \right\} \\
        &\cup \left\{ (a, b) \in \mathbb{N}^{2} \  \middle| \
        3 \leq a \leq \left\lfloor\frac{n - 1}{2}\right\rfloor,\  
        \max \left\{ a, \frac{n-a}{2} \right\} < b \leq n - a \right\}, 
      \end{aligned}
    \end{displaymath}
    where
    \begin{displaymath}
      {\rm CW}_{2, \dim} (n) = \left\{
      \begin{alignedat}{3}
        &\left\{ (2, n - 2), (2, n - 3) \right\},
        &\quad \text{if $n$ is even}, \\
        &\left\{ (2, n - 2), (2, n - 3), \left(2, \frac{n - 1}{2}\right) \right\}, 
        &\quad \text{if $n$ is odd}.
      \end{alignedat}
      \right.
    \end{displaymath}
\end{theorem}

As a corollary, we have:
\begin{corollary}
  \label{CWddConvex}
  The set ${\rm CW}_{\depth, \dim} (n)$ is convex
  if and only if $n$ is even or $n=5, 7$. 
\end{corollary}

We give some examples. 
\begin{example}
  By virtue of Theorem \ref{CWdd}, one can determine
  ${\rm CW}_{\depth,\dim}(n)$ for all $n$.  In Figure
  \ref{figurecw89}, we have plotted the elements
  of ${\rm CW}_{\depth,\dim}(n)$ for $n=8$ and $n=9$.  Observe
  that when $n=9$, the set is not convex.

\begin{figure}
\begin{tikzpicture}[thick,scale=0.9, every node/.style={scale=0.9}]
\draw[thin,->] (5.5-4,0) -- (8.5-4,0) node[right] {${\rm depth}$};
\draw[thin,->] (6-4,-0.5) -- (6-4,4.5) node[above] {$\dim$};

\foreach \x [count=\xi starting from 0] in {1,2,3,4,5,6,7,8}{
    \draw (5.9-4,\x/2) -- (6.1-4,\x/2);
    \ifodd\xi
    \fi
    \draw (6+1/2-4,1pt) -- (6+1/2-4,-1pt);
    \draw (6+2/2-4,1pt) -- (6+2/2-4,-1pt);
    \draw (6+3/2-4,1pt) -- (6+3/2-4,-1pt);
    \draw (6+4/2-4,1pt) -- (6+4/2-4,-1pt);
     \node[anchor=north] at (6+1/2-4,0) {$1$};
     \node[anchor=north] at (6+2/2-4,0) {$2$};
     \node[anchor=north] at (6+3/2-4,0) {$3$};
     \node[anchor=north] at (6+4/2-4,0) {$4$};
}

\
\foreach \point in {
  (6+2/2-4,5/2),(6+2/2-4,6/2),
  (6+3/2-4,3/2),(6+3/2-4,4/2),(6+3/2-4,5/2)
}
  {
    \fill \point circle (2pt);
}

\draw[thin,->] (5.5+1,0) -- (8.5+1.5,0) node[right] {${\rm depth}$};
\draw[thin,->] (6+1,-0.5) -- (6+1,4.5) node[above] {$\dim$};

\foreach \x [count=\xi starting from 0] in {1,2,3,4,5,6,7,8}{
    \draw (5.9+1,\x/2) -- (6.1+1,\x/2);
    \ifodd\xi
    \fi
\draw (6+1/2+1,1pt) -- (6+1/2+1,-1pt);
 \draw (6+2/2+1,1pt) -- (6+2/2+1,-1pt);
 \draw (6+3/2+1,1pt) -- (6+3/2+1,-1pt);
 \draw (6+4/2+1,1pt) -- (6+4/2+1,-1pt);
  \draw (6+5/2+1,1pt) -- (6+5/2+1,-1pt);
     \node[anchor=north] at (6+1/2+1,0) {$1$};
     \node[anchor=north] at (6+2/2+1,0) {$2$};
     \node[anchor=north] at (6+3/2+1,0) {$3$};
     \node[anchor=north] at (6+4/2+1,0) {$4$};
     \node[anchor=north] at (6+5/2+1,0) {$5$};
}

\
\foreach \point in {
  (
  (6+2/2+1,4/2),
  (6+2/2+1,6/2),
  (6+2/2+1,7/2),
  (6+3/2+1,4/2),
  (6+3/2+1,5/2),
  (6+3/2+1,6/2),
  (6+4/2+1,4/2),(6+4/2+1,5/2)
}
  {
    \fill \point circle (2pt);
  }
  
\node at (8-4,4) {$n=8$};
\node at (8+2,4) {$n=9$};
\end{tikzpicture}
\caption{${\rm CW}_{{\rm depth},\dim}(n)$ for
  $n=8,9$}\label{figurecw89}
\end{figure}
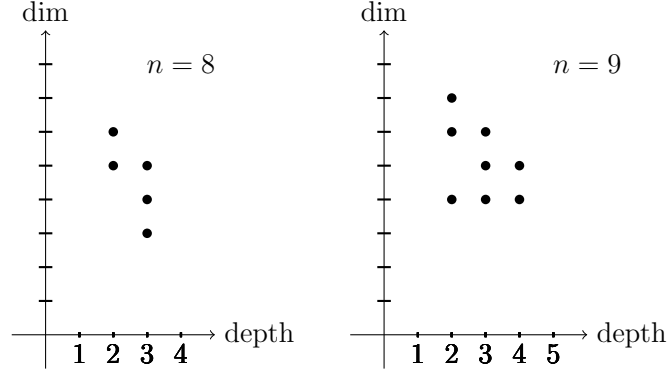
%
%
\end{example}

We close this section by proving Theorem \ref{CWdd}. 
\begin{proof}[{Proof of Theorem \ref{CWdd}}]
  ($\subseteq$) : Let $(a, b) \in {\rm CW}_{{\rm depth}, \dim}(n)$.
  Then there exists a Cameron--Walker graph $G$ with $|V(G)| = n$
  such that ${\rm depth}(R/I(G)) = a$ and $\dim R/I(G) = b$.
  Then one has $a \leq b$. Also note that $a \geq 2$ by Lemma \ref{depth2CW}. 

  \par
  If $a=2$, then one has $(a,b) \in {\rm CW}_{2, \dim} (n)$
  by Proposition \ref{depth2}. 

  \par
  Assume that $3 \leq a = b$. Then $R/I(G)$ is Cohen-Macaulay. 
  Hence one has $n = 2m + 3p$ and $b = m + p$ with the notation as in
  Figure \ref{fig:CameronWalkerGraph} by Lemma \ref{CWDDLemma} (iv).
  Thus $n/3 < b < n/2$. 

  \par
  Assume that $3 \leq a < b$. 
  Then one has $3 \leq a \leq \left\lfloor \frac{n - 1}{2} \right\rfloor$,
  $a < b \leq n - a$ and $(n-a)/2 < b$ by virtue of
  Lemma \ref{CWDDLemma} (i), (ii) and (iii).  

  \par
  Now we have the desired conclusion. 

  \par
  ($\supseteq$) :
  The inclusion ${\rm CW}_{{\rm depth}, \dim} (n) \supseteq {\rm CW}_{2, \dim} (n)$
  follows by Proposition \ref{depth2}. 

  \par
  Let $b$ be an integer with $n/3 < b < n/2$.
  Then the graph $G^{(1)}_{3b-n, n-2b, 1}$ which appears
  in Construction \ref{const:specialCW1} guarantees 
  $(b,b) \in {\rm CW}_{{\rm depth}, \dim} (n)$. 
  Indeed, 
  \begin{itemize}
  \item $|V(G^{(1)}_{3b-n, n-2b, 1})| = 2(3b - n) + 3(n - 2b) + 2 - 2 = n$, 
  \item ${\rm depth}(R/I(G^{(1)}_{3b-n, n-2b, 1})) = (3b - n) + (n - 2b) = b$, and
  \item $\dim R/I(G^{(1)}_{3b-n, n-2b, 1}) = (3b - n) + (n - 2b) + 1 - 1 = b$ 
  \end{itemize}
  by virtue of Lemma \ref{ddspecialCW1}. 
  Hence $\{(b, b) \in \mathbb{N}^{2} \mid n/3 < b < n/2 \}
  \subseteq {\rm CW}_{{\rm depth}, \dim}(n)$. 

  \par
  Next, let $a, b$ be integers such that
  $3 \leq a \leq \left\lfloor\frac{n - 1}{2}\right\rfloor$,
  $\max \{ a, (n-a)/2 \} < b \leq n - a$.
  We distinguish with $3$ cases:
  $b<n/2$; $n/2 \leq b < n-a$; $b =  n-a$. 
  
  \par
  If $b < n/2$, 
  then the graph $G^{(1)}_{a+2b-n,  n-2b, b-a+1}$ which appears
  in Construction \ref{const:specialCW1} guarantees
  $(a,b) \in {\rm CW}_{{\rm depth}, \dim}(n)$. 
  Indeed Lemma \ref{ddspecialCW1} says that
  \begin{itemize}
  \item $|V(G^{(1)}_{a+2b-n,  n-2b, b-a+1})|
    = 2(a + 2b - n) + 3(n - 2b) + 2(b - a + 1) - 2 = n$, 
  \item ${\rm depth}(R/I(G^{(1)}_{a+2b-n,  n-2b, b-a+1}))
    = (a + 2b - n) + (n - 2b) = a$, and
  \item $\dim R/I(G^{(1)}_{a+2b-n,  n-2b, b-a+1})
    = (a + 2b - n) + (n - 2b) + (b - a + 1) - 1 = b$. 
  \end{itemize}

  \par
  If $n/2 \leq b < n -a$, 
  then the graph $G^{(2)}_{a-1, -n+2b+1, n-a-b}$ which appears
  in Construction \ref{const:specialCW2} guarantees 
  $(a,b) \in {\rm CW}_{{\rm depth}, \dim}(n)$. 
  Indeed Lemma \ref{ddspecialCW2} says that
  \begin{itemize}
  \item $|V(G^{(2)}_{a-1, -n+2b+1, n-a-b})|
    = 2(a - 1) + (-n + 2b + 1) + 2(n - a - b) + 1 = n$, 
  \item ${\rm depth}(R/I(G^{(2)}_{a-1, -n+2b+1, n-a-b})) = (a - 1) + 1 = a$, and
  \item $\dim R/I(G^{(2)}_{a-1, -n+2b+1, n-a-b})
    = (a - 1) + (-n + 2b + 1) + (n - a - b) = b$.
  \end{itemize}
  
  \par
  If $b = n-a$, 
  then the Cameron--Walker graph $G$ such that 
  $m = a$, $p = 1$, $s_{i} = 1$ ($1 \leq i \leq a - 1$),
  $s_{a} = b - a$ and $t_{1} = 0$ with the notation as
  in Figure \ref{fig:CameronWalkerGraph} guarantees 
  $(a,b) \in {\rm CW}_{{\rm depth}, \dim}(n)$. 
  Indeed, by virtue of Theorem \ref{CWformulas}, one has
  \begin{itemize}
  \item $|V(G)| = a + 1 + (b - 1) = a + b = n$, 
  \item $\dim R/I(G) = (b - 1) + 1 = b$, and
  \item ${\rm depth}(R/I(G)) = \min\{ b, a \} = a$. 
  \end{itemize}

  \par
  Hence one has 
  \[
  \left\{ (a, b) \in \mathbb{N}^{2} \  \middle| \
  3 \leq a \leq \left\lfloor\frac{n - 1}{2}\right\rfloor,
  \max \left\{ a,  \frac{n-a}{2} \right\} < b \leq n - a \right\} 
  \subseteq {\rm CW}_{{\rm depth}, \dim}(n).  
  \]
%

  \par
  This completes the proof. 
\end{proof}

\medskip

\section{Lattice points of ${\rm CW}_{{\rm depth}, {\rm reg}, \dim, \deg h}(n)$}
\label{sec:MainResult}
Let $n \geq 5$ be an integer. 
In previous section, we determined 
the set ${\rm CW}_{{\rm depth}, \dim} (n)$, which is the set of
all possible pairs $(\depth (R/I(G)), \dim R/I(G))$ arising from
$G \in {\rm CW}(n)$. 
On the other hand, in \cite[Theorem 5.1]{HKMVT}, the authors determined
the set ${\rm CW}_{{\rm reg}, \deg h} (n)$, which is the set of 
all possible pairs $(\reg R/I(G), \deg h(R/I(G)))$ arising from
$G \in {\rm CW}(n)$. 
\begin{theorem}[{\cite[Theorem 5.1]{HKMVT}}]
  \label{RD}
  Let $n \geq 5$ be an integer. Then $(r,d) \in {\rm CW}_{{\rm reg}, \deg h} (n)$
  if and only if 
  \begin{displaymath}
    2 \leq r \leq \left\lfloor \frac{n-1}{2} \right\rfloor
    \quad \text{and} \quad 
    \max \{ r, -2r+n+1 \} \leq d \leq n-r.
  \end{displaymath}
\end{theorem}

\par
In this section, we determine
\[
{\rm CW}_{{\rm depth}, {\rm reg}, \dim, \deg h}(n)
\]
which is the set of all possible tuples
\[
(\depth (R/I(G)), \reg (R/I(G)), \dim R/I(G), \deg h(R/I(G)))
\]
arising from $G \in {\rm CW}(n)$ (Theorem \ref{maintheorem3}).


We note that by virtue of Theorem \ref{CWformulas} and Lemma \ref{depth2CW},
we know if
$(a,r,b,d) \in {\rm CW}_{{\rm depth}, {\rm reg}, \dim, \deg h}(n)$,
then $2 \leq a \leq r \leq b=d$.
In particular, an element belonging to ${\rm CW}_{{\rm depth}, {\rm reg}, \dim, \deg h}(n)$
is always of the form $(a,r,d,d)$ with $2 \leq a \leq r \leq d$. 
In order to determine ${\rm CW}_{{\rm depth}, {\rm reg}, \dim, \deg h}(n)$,
we need more information about the relations among the homological invariants
of $I(G)$ for a Cameron--Walker graph $G$. 

\begin{lemma}\label{CWdrdd}
Let $G$ be a Cameron--Walker graph. 
Then
\begin{enumerate}
\item[$(i)$] $|V(G)| + 1
  \leq {\rm depth}(R/I(G)) + {\rm reg}(R/I(G)) + \dim R/I(G)$. 
\item[$(ii)$] Assume that
  $|V(G)| + 1 = {\rm depth}(R/I(G)) + {\rm reg}(R/I(G)) + \dim R/I(G)$ 
  and ${\rm depth}(R/I(G)) < {\rm reg}(R/I(G))$. 
  Then ${\rm reg}(R/I(G)) = \dim R/I(G)$. 
\end{enumerate}
\end{lemma}
\begin{proof}
  We use the notation as in Figure \ref{fig:CameronWalkerGraph}. 
  \begin{enumerate}
  \item[$(i)$] By virtue of \cite[Theorem 5.2]{HKMVT} and
    Theorem \ref{CWformulas} (iii), one has 
    \begin{displaymath}
      \begin{aligned}
        &|V(G)| - {\rm reg}(R/I(G)) - \dim R/I(G) \\
        &= |\{j : t_{j} > 0\}| 
        \leq {\rm depth}(R/I(G)) - m \leq {\rm depth}(R/I(G)) - 1. 
      \end{aligned}
    \end{displaymath}
      Hence we have the desired conclusion. 
    \item[$(ii)$] Assume that
      $|V(G)| + 1 = {\rm depth}(R/I(G)) + {\rm reg}(R/I(G)) + \dim R/I(G)$. 
      Then, by the argument of (i), we have $m = 1$. 
      Hence we may assume that $t_{j} > 0$ for all $1 \leq j \leq p$.
      Then by Theorem \ref{CWformulas} (ii), (v), one has
      ${\rm reg}(R/I(G)) = 1 + \sum_{j = 1}^{p}t_{j}$ and 
      $\dim R/I(G) = s_{1} + \sum_{j = 1}^{p}t_{j}$.
      Thus in order to prove (ii), it is sufficient to show $s_1 = 1$. 

      \par
      In this case, the bipartite part of $G$ is a complete bipartite
      graph.
      Hence Theorem \ref{CWformulas} (iv), (v) and the assumption
      that $\depth (R/I(G)) < \reg (R/I(G))$ implies that
      ${\rm depth}(R/I(G)) = \min\{1 + \sum_{j = 1}^{p}t_{j}, p + s_{1}\}
        = p + s_{1}$. 
      Also by Theorem \ref{CWformulas} (i), one has
      $|V(G)| = 1 + p + s_{1} + 2\sum_{j = 1}^{p}t_{j}$.
      Thus it follows that 
      \begin{eqnarray*}
      p &=& |V(G)| - {\rm reg}(R/I(G)) - \dim R/I(G) \\ 
      &=& {\rm depth}(R/I(G)) - 1 \\
      &=& p + s_{1} - 1. 
      \end{eqnarray*}
      Hence $s_{1} = 1$, as desired. 
  \end{enumerate}
\end{proof}
\begin{remark}
  The proof of Lemma \ref{CWdrdd} (ii) implies that
  a Cameron--Walker graph satisfying the assumption (ii)
  of Lemma \ref{CWdrdd} must have the form 
  $m=1$, $s_1 = 1$, and $t_j > 0$ for all $1 \leq j \leq p$.
\end{remark}


\par
Now we come to the main theorem in this paper.

\begin{theorem}\label{maintheorem3}
Let $n \geq 5$ be an integer. Then
\begin{displaymath}
  \begin{aligned}
    {\rm CW}_{{\rm depth}, {\rm reg}, \dim, \deg h} (n) 
    &= {\rm CW}_{2, {\rm reg}, \dim, \deg h} (n) \\
    &\cup \left\{ (a, d, d, d) \in \mathbb{N}^{4} \  \middle| \
    3 \leq a \leq d \leq \left\lfloor\frac{n - 1}{2}\right\rfloor, \  
    n < a + 2d \right\} \\ 
    &\cup \left\{ (a, a, d, d) \in \mathbb{N}^{4} \  \middle| \
    3 \leq a < d \leq n - a,\   n \leq 2a + d - 1 \right\} \\
    &\cup \left\{(a,r,d,d) \in \mathbb{N}^{4} ~\left|~
    \begin{array}{c}
      \mbox{$3 \leq a < r < d < n - r$, } \\
      \mbox{$n + 2 \leq a + r + d$} \\ 
    \end{array}
    \right \}\right., 
  \end{aligned}
\end{displaymath}
where
\begin{displaymath}
  \begin{aligned}
    &{\rm CW}_{2, {\rm reg}, \dim, \deg h} (n) \\
    &= \left\{
    \begin{alignedat}{3}
      &\{(2, 2, n - 2, n - 2), (2, 2, n - 3, n - 3)\}, 
      &\quad &\text{if $n$ is even}, \\
      & \left\{(2, 2, n - 2, n - 2), (2, 2, n - 3, n - 3),
      \left(2, \frac{n-1}{2}, \frac{n-1}{2}, \frac{n-1}{2}\right) \right\}, 
      &\quad &\text{if $n$ is odd}. 
    \end{alignedat}
    \right. 
  \end{aligned}
\end{displaymath}
\end{theorem}

\begin{proof}
  ($\subseteq$) \
  Take an element of ${\rm CW}_{{\rm depth}, {\rm reg}, \dim, \deg h}(n)$.
  As noted in the beginning of this section, it is of the form
  $(a,r,d,d)$ with $2 \leq a \leq r \leq b$.
  Let $G$ be a Cameron--Walker graph with $|V(G)| = n$, 
  ${\rm depth}(R/I(G)) = a$, ${\rm reg}(R/I(G)) = r$,
  and $\dim R/I(G) = \deg h(R/I(G)) = d$. 
  We distinguish the proof with $4$ cases:
  $a=2$; $3 \leq a \leq r=d$; $3 \leq a = r < d$; and 
  $3 \leq a < r <d$.
  
  \par
  First consider the case $a=2$. In this case,
  we have $(a, r, d, d) \in {\rm CW}_{2, {\rm reg}, \dim, \deg h}(n)$
  by Lemma \ref{depth2Invariants}. 

  \par
  Second assume that $3 \leq a \leq r = d$.
  Then Theorem \ref{RD} says that
  $3 \leq a \leq d \leq \left\lfloor \frac{n-1}{2} \right\rfloor$. 
  Moreover, one has $n < a + 2d$ by Lemma \ref{CWDDLemma} (iii). 

  \par
  Third assume that $3 \leq a = r < d$.
  Then Theorem \ref{RD} says that $n \leq 2r + d - 1 = 2a + d - 1$. 
  Moreover, one has $d \leq n - a$ by Lemma \ref{CWDDLemma} (ii).

  \par
  Finally, assume that $3 \leq a < r < d$. 
  Note that $d \leq n - r$ by Theorem \ref{RD}. 
  Suppose that $d = n - r$. Then $t_{j} = 0$ for all $1 \leq j \leq p$
  by virtue of \cite[Theorem 5.2]{HKMVT}. 
  Then \cite[Corollary 2.4]{HKMT} says that $a = r$,
  but this is a contradiction. Thus we have $d < n - r$. 
  Moreover, one has $n + 2 \leq a + r + d$ by Lemma \ref{CWdrdd}.

  \par
  Therefore we have the desired conclusion. 

  \par
  $(\supseteq)$\
  First we consider the inclusion
  ${\rm CW}_{2, {\rm reg}, \dim, \deg h}(n)
  \subseteq {\rm CW}_{{\rm depth}, {\rm reg}, \dim, \deg h}(n)$. 
  The graphs provided in the proof of Proposition \ref{depth2}
  guarantee the inclusion. 

  \par
  Second, let $a, d$ be integers with
  $3 \leq a \leq d \leq \left\lfloor\frac{n - 1}{2}\right\rfloor$
  and $n < a + 2d$. 
  Consider the graph $G^{(1)}_{a+2d-n, n-2d, d-a+1}$ which appears
  in Construction \ref{const:specialCW1}.
  Note that $d \leq \left\lfloor\frac{n - 1}{2}\right\rfloor$ implies
  $n-2d \geq 1$. Then we have
  \begin{itemize}
  \item $|V(G)| = 2(a+2d-n) + 3(n-2d) + 2(d-a+1) - 2 = n$, 
  \item ${\rm depth}(R/I(G)) = (a+2d-n) + (n-2d) = a$, and  
  \item $\dim R/I(G) = {\rm reg}(R/I(G)) = (a+2d-n) + (n-2d) + (d-a+1) -1 = d$ 
  \end{itemize}
  by virtue of Lemma \ref{ddspecialCW1}. Hence one has 
  $(a, d, d, d) \in {\rm CW}_{{\rm depth}, {\rm reg}, \dim, \deg h}(n)$, 
  and thus 
  \[
  \left\{ (a, d, d, d) \in \mathbb{N}^{4} \  \middle| \
  3 \leq a \leq d \leq \left\lfloor\frac{n - 1}{2}\right\rfloor, \, 
  n < a + 2d \right\} \subseteq {\rm CW}_{{\rm depth}, {\rm reg}, \dim, \deg h}(n). 
  \]

  \par
  Third, let $a, d$ be integers with $3 \leq a < d \leq n - a$
  and $n \leq 2a + d - 1$.
  
  \par
  When $d < n - a$, we consider the Cameron--Walker graph $G$
  such that its bipartite part is a complete bipartite graph, 
  $m = 2a + d - n$, $p = n - a - d$,
  $s_i = 1$ for all $1 \leq i \leq m-1$, $s_{m} = d-a+1$, and 
  $t_{j} = 1$ for all $1 \leq j \leq p$
  with notation as in Figure \ref{fig:CameronWalkerGraph}. 
  Then Theorem \ref{CWformulas} says that 
  \begin{itemize}
  \item $|V(G)| = (2a + d - n) + (n - a - d) + (a + 2d - n) + 2(n - a - d) = n$, 
    \item ${\rm depth}(R/I(G)) = \min\{ (2a + d - n) + (n - a - d), (n - a - d) + (a + 2d - n) \} = \min\{a, d\} = a$,
    \item ${\rm reg}(R/I(G)) = (2a + d - n) + (n - a - d) = a$, and  
    \item $\dim R/I(G) = (a + 2d - n) + (n - a - d) = d$.
  \end{itemize}
  Hence one has 
  $(a, a, d, d) \in {\rm CW}_{{\rm depth}, {\rm reg}, \dim, \deg h}(n)$.

  \par
  When $d = n - a$, we consider the Cameron--Walker graph $G$ such that 
  its bipartite part is a complete bipartite graph, 
  $m = a$, $p = d - a$, 
  $s_{i} = 1$ for all $1 \leq i \leq m$
  and $t_{j} = 0$ for all $1 \leq j \leq p$
  with notation as in Figure \ref{fig:CameronWalkerGraph}. 
Then Theorem \ref{CWformulas} says that
\begin{itemize}
    \item $|V(G)| = a + (d - a) + a = a + d = n$,
    \item ${\rm depth}(R/I(G)) = \min\{a, d - a + a\} = \min\{a, d\} = a$, 
    \item ${\rm reg}(R/I(G)) = a$, and 
    \item $\dim R/I(G) = a + (d - a) = d$. 
\end{itemize}
  Hence one has 
  $(a, a, d, d) \in {\rm CW}_{{\rm depth}, {\rm reg}, \dim, \deg h}(n)$.

  \par
  Therefore we have
\[
\left\{ (a, a, d, d) \in \mathbb{N}^{4} \  \middle| \
3 \leq a < d \leq n - a, \, n \leq 2a + d - 1 \right\}
\subseteq {\rm CW}_{{\rm depth}, {\rm reg}, \dim, \deg h}(n). 
\]

\par
Finally, let $a, r, d$ be integers with $3 \leq a < r < d < n - r$
and $n+2 \leq a + r + d$. 

\par
When $n \geq a + 2r$, we consider the Cameron--Walker graph $G$ such that
\begin{itemize}
    \item $m = 2$ and $p = n - d - r + 1 (\geq 2)$, 
    \item $E(G_{\rm bip}) = \left\{ \{v_{1}, w_{1}\}, \{v_{1}, w_{2}\}, \ldots,
      \{v_{1}, w_{n - d - r + 1}\}, \{v_{2}, w_{n - d - r + 1}\} \right\}$, 
    \item $s_{1} = a + r + d - n - 1 (\geq 1)$
      and $s_{2} = n - a - 2r + 2 (\geq 2)$, and 
    \item $t_{1} = 2r + d - n - 1 (> s_{1})$,
      $t_{j} = 1$ for all $2 \leq j \leq n - d - r$ and $t_{n-d-r+1} = 0$ 
\end{itemize}
  with notation as in Figure \ref{fig:CameronWalkerGraph}. 
Then Theorem \ref{CWformulas} says that 
\begin{itemize}
    \item $|V(G)| = m + p + s_{1} + s_{2} + 2\sum_{j = 1}^{p}t_{j} = n$, 
    \item ${\rm reg}(R/I(G)) = m + \sum_{j = 1}^{p}t_{j} = r$, and 
    \item $\dim R/I(G)
      = s_{1} + s_{2} + \sum_{j = 1}^{p}t_{j} + |\{j:t_{j} = 0\}| = d$. 
\end{itemize}
We prove ${\rm depth}(R/I(G)) = a$. 
Then we have 
  $(a, r, d, d) \in {\rm CW}_{{\rm depth}, {\rm reg}, \dim, \deg h}(n)$.

\par
Let $f(V)$ be the function which appears in Theorem \ref{CWdepthV}. 
Then $f(\emptyset) = r > a$, $f(\{v_{1}\}) = a$,
$f(\{v_{2}\}) = n - a - r + 1 > a$ and $f(\{v_{1}, v_{2}\}) = n - 2r + 2 > a$. 
Thus one has ${\rm depth}(R/I(G)) = a$. 

\par
When $n < a + 2r$, we consider the Cameron--Walker graph $G$ such that 
\begin{itemize}
    \item $m = a + 2r + 1 - n (\geq 2)$ and $p = n - d - r + 1 (\geq 2)$.  
    \item The edge set of the bipartite part is 
    \begin{eqnarray*} 
      E(G_{\rm bip}) &=& \left\{ \{v_{1}, w_{1}\}, \{v_{1}, w_{2}\}, \ldots,
      \{v_{1}, w_{n - d - r + 1}\} \right\} \\
      &\cup& \left\{ \{v_{2}, w_{n - d - r + 1}\}, \{v_{3}, w_{n - d - r + 1}\},
      \ldots, \{v_{a + 2r + 1 - n}, w_{n - d - r + 1}\} \right\}.
    \end{eqnarray*}
  \item $s_{1} = d - r (\geq 1)$ and $s_{i} = 1$ for all
    $2 \leq i \leq a + 2r + 1 - n$. 
  \item $t_{1} = d - a (> s_{1})$, $t_{j} = 1$ for all
    $2 \leq j \leq n - d - r$ and $t_{n-d-r+1} = 0$ 
\end{itemize}
  with notation as in Figure \ref{fig:CameronWalkerGraph}. 
Then Theorem \ref{CWformulas} says that 
\begin{itemize}
    \item $|V(G)| = m + p + \sum_{i=1}^{m}s_{i} + 2\sum_{j = 1}^{p}t_{j} = n$, 
    \item ${\rm reg}(R/I(G)) = m + \sum_{j = 1}^{p}t_{j} = r$, and
    \item $\dim R/I(G) = \sum_{i = 1}^{m}s_{i} + \sum_{j = 1}^{p}t_{j} + |\{j:t_{j} = 0\}| = d$. 
\end{itemize}
We prove ${\rm depth}(R/I(G)) = a$. 
Then we have 
  $(a, r, d, d) \in {\rm CW}_{{\rm depth}, {\rm reg}, \dim, \deg h}(n)$.

\par
We use Theorem \ref{CWdepthV} again. Let $V$ be a subset of
$\{ v_1, \ldots, v_s \}$. 
If $v_{1} \not\in V$, then $f(V) = r > a$. 
If $v_{1} \in V$ and $V \neq \{v_{1}, \ldots, v_{a + 2r + 1 - n}\}$, then $f(V) = a$. 
Moreover, $f(\{v_{1}, \ldots, v_{a+2r+1-n}\}) = a + 1$. 
Hence we have ${\rm depth}(R/I(G)) = a$. 

\par
Therefore one has 
\[
\left\{(a,r,d,d) \in \mathbb{N}^{4} ~\left|~
\begin{array}{c}
  \mbox{$3 \leq a < r < d < n - r$, } \\
  \mbox{$n + 2 \leq a + r + d$} \\ 
\end{array}
\right \}\right. 
\subseteq {\rm CW}_{{\rm depth}, {\rm reg}, \dim, \deg h}(n). 
\]

\end{proof}

\begin{example}
By virtue of Theorem \ref{maintheorem3}, we have 
\[
{\rm CW}_{{\rm depth}, {\rm reg}, \dim, \deg h}(8) = \{ (2,2,5,5), (2,2,6,6), (3,3,3,3), (3,3,4,4), (3,3,5,5)\}. 
\]
and
\begin{eqnarray*}
{\rm CW}_{{\rm depth}, {\rm reg}, \dim, \deg h}(9) = \{ (2,2,6,6), (2,2,7,7), (2,4,4,4), (3,4,4,4), \\ 
(4,4,4,4), (3,3,4,4), (3,3,5,5), (3,3,6,6), (4,4,5,5)\}. 
\end{eqnarray*}
\end{example}


\bibliographystyle{plain}

\newpage

\end{document}